\documentclass[12pt]{amsart}
\usepackage[dvips]{graphicx}
\usepackage{amsmath,xypic,graphics,comment}
\usepackage{amsfonts,amssymb,amscd,color}
\theoremstyle{plain}
\newtheorem{theorem*}{Theorem}
\newtheorem*{lemma*}{Lemma}
\newtheorem{corollary*}{Corollary}
\newtheorem*{proposition*}{Proposition}
\newtheorem{conjecture*}{Conjecture}

\newtheorem*{prop*}{Proposition}

\newtheorem{theorem}{Theorem}[section]
\newtheorem{lemma}[theorem]{Lemma}
\newtheorem{corollary}[theorem]{Corollary}
\newtheorem{proposition}[theorem]{Proposition}

\newtheorem*{thm*}{Theorem}

\theoremstyle{remark}
\newtheorem{definition}[theorem]{Definition}

\newtheorem*{example*}{Example}
\newtheorem*{claim}{Claim}

\theoremstyle{definition}

\newtheorem*{defn*}{Definition}

\def\SFH{\mathit{SFH}}
\def\op{\operatorname}
 \def\Q{\Bbb{Q}} \def\F{\Bbb{F}} \def\Z{\Bbb{Z}} \def\R{\Bbb{R}} 
\def\N{\Bbb{N}}  \def\l{\lambda}  
 \def\a{\alpha} \def\g{\gamma}  \def\bp{\begin{pmatrix}}
\def\sm{\setminus} \def\ep{\end{pmatrix}} \def\bn{\begin{enumerate}} 
   \def\en{\end{enumerate}}
\def\ba{\begin{array}} \def\ea{\end{array}}  
   \def\a{\alpha} \def\b{\beta} \def\ti{\tilde}
  \def\im{\op{Im}} 
  
\def\be{\begin{equation}} \def\ee{\end{equation}} 
   
 \def\hom{\op{Hom}}  
   \def\eps{\epsilon}
 
\def\zx{\Z[x^{\pm 1}]}    \def\rk{\op{rank}}

\def\spin{\op{Spin}}
\def\supp{\op{Supp}}
\def\spinc{\operatorname{Spin}^c}

\def\min{\op{min}}

\def\s{\mathfrak{s}}
\def\t{\mathfrak{t}}

\newcommand{\bZ}{{\mathbb Z}}

\newcommand{\fs}{\mathfrak{s}}
\newcommand{\ft}{\mathfrak{t}}

\newcommand{\mF}{\mathcal{F}}

\newcommand{\De}{\Delta}

\newcommand{\la}{\lambda}
\newcommand{\al}{\alpha}

\newcommand{\ga}{\gamma}
\def\MM{\mathcal{M}}
\def\CC{\mathcal{C}}
\def\FF{\mathcal{F}}

\newcommand{\bdd}{\partial}

\begin{document}

\title[Sutured Floer homology, fibrations, and foliations]{Sutured Floer homology, fibrations, and taut depth one foliations}
\author{Irida Altman}
\email{irida.altman@gmail.com}
\author{Stefan Friedl}
\address{Fakult\"at f\"ur Mathematik, Universit\"at Regensburg, 93040 Regensburg, Germany }
\email{sfriedl@gmail.com}
\author{Andr\'as Juh\'asz}
\thanks{AJ was supported by a Royal Society Research Fellowship and OTKA grant NK81203}
\address{Mathematical Institute, University of Oxford, Andrew Wiles Building, Radcliffe Observatory Quarter, Woodstock Road, Oxford, OX2 6GG, UK}
\email{juhasza@maths.ox.ac.uk}

\begin{abstract}
For an oriented irreducible 3-manifold~$M$ with non-empty toroidal boundary, we describe how sutured Floer homology ($\SFH$)
can be used to determine all fibred classes in $H^1(M)$.
Furthermore, we show that the $\SFH$ of a balanced sutured manifold $(M,\g)$
detects which classes in $H^1(M)$ admit a taut depth one foliation such that the only compact
leaves are the components of~$R(\g)$. The latter had been proved earlier by the first author under
the extra assumption that~$H_2(M)=0$. The main technical result is that we can obtain an extremal
$\spinc$-structure~$\s$ (i.e., one that is in a `corner' of the support of $\SFH$)
via a nice and taut sutured manifold decomposition even when~$H_2(M) \neq 0$,
assuming the corresponding group $SFH(M,\g,\s)$ has non-trivial Euler characteristic.
\end{abstract}
\maketitle

\section{Introduction}

Ozsv\'ath and Szab\'o~\cite{OS04a, OS04b} defined Heegaard Floer homology, a package of invariants for closed $3$-manifolds.
This was later extended by them~\cite{OS04c, OS08} and Rasmussen~\cite{Ras03} to knots and links. Sutured Floer homology, introduced by
the third author~\cite{Ju06}, is a common generalisation of the hat versions of Heegaard Floer homology and knot Floer homology
to $3$-manifolds with boundary decorated with a dividing set, called sutured manifolds.
In particular, the invariant assigned to the sutured manifold complementary to a link is link Floer homology.

Ghiggini~\cite{Gh08}, Ni~\cite{Ni07, Ni09a}, and the third author~\cite{Ju08, Ju10} showed that knot Floer homology detects fibred knots.
Furthermore Ni~\cite{Ni09} proved that~$HF^+$ detects whether a closed irreducible 3-manifold fibres over~$S^1$
if the genus of the fibre is greater than one.
Also, Ai and Yi~\cite{AN09} showed that Heegaard Floer homology with twisted coefficients detects torus bundles.
It is now natural to ask if Heegaard Floer type invariants can also detect whether an arbitrary
3-manifold with non-empty boundary is fibred (in which case the boundary necessarily has to be toroidal).
If~$M$ is an irreducible $3$-manifold such that  the map $H_2(\partial M;\Q) \to H_2(M;\Q)$ is surjective, that is, if~$M$ is an irreducible $3$-manifold which is the exterior of a link in a rational homology sphere,
then, as we shall see in Proposition~\ref{prop:fibred}, it is straightforward that the methods of~\cite{Ju08, Ju10}
imply that sutured Floer homology detects whether or not~$M$ is fibred.
In this paper, we show that this restriction on~$M$ can be dropped, in particular,  we will prove the following theorem
(for a precise formulation, we refer the reader to Section~\ref{section:fibred}).

\begin{theorem}\label{mainthm-fib}
Sutured Floer homology detects whether or not an irreducible 3-manifold~$M$ with non-empty boundary is fibred,
and determines all fibred classes in~$H^1(M)$.
\end{theorem}

In order to state our second main result we first recall that a \emph{sutured manifold  $(M,\gamma)$} is loosely speaking a $3$-manifold $M$, together with a decomposition of its boundary
\[ \partial M=(-R_-)\cup \g\cup R_+,\]
where $\g$ is a collection  of annuli, and $R_-$ and $R_+$ are oriented  subsurfaces of $\partial M$.
In the following we will, for the most part, restrict ourselves to \emph{taut balanced sutured manifolds};
these are sutured manifolds such that $M$ is irreducible, and $R_\pm$ have minimal complexity and
no closed components. We refer to Section~\ref{section:defsutured} for details.

In this paper, we will also study the existence of taut foliations on a sutured manifold. By a \emph{foliation on a sutured manifold $(M,\g)$},
we mean a codimension~1 transversely oriented
foliation on the 3-manifold~$M$ such that all the leaves of~$\mF$ are transverse to~$\ga$ and tangential to $R_-$ and $R_+$.
We say that a foliation $\mF$ is \emph{taut} if  there exists a curve or properly embedded arc in $M$ that is transverse to the leaves of $\mF$ and that intersects every leaf of $\mF$ at least once.

A leaf~$L$ of a foliation~$\mF$ of~$(M,\g)$ is said to be {\it depth~$0$} if it is compact. Recursively, we now define a leaf to be of {\it depth~$k$}
if it is not of depth less than $k$,  and if $\overline L \setminus L$ is a union of leaves of depth less than~$k$. Otherwise, we say that~$L$ is of depth~$\infty$. A foliation~$\mF$ of~$(M,\ga)$ is of depth~$k$ if the maximum of the depths of its leaves is~$k$.
In this paper, we say that a foliation of depth greater than 0 is  \emph{indecomposable}
if the components of $R_\pm$ are the only compact leaves.
We refer to Section~\ref{section:deffoliation} for details.

The third author~\cite{Ju06}
introduced the notion of a relative $\spinc$-structure on a balanced sutured manifold $(M,\g)$, and associated to a $\spinc$-structure $\s$, an abelian group $\SFH(M,\g,\s)$ called the {\it sutured Floer homology} of $(M,\g)$ at $\fs$.
Let $\supp(M,\g)$ denote the set of those $\spinc$-structures on $(M,\ga)$ for which
$\SFH(M,\g,\s) \neq 0$. We write
\[ \SFH(M,\g):= \bigoplus_{\s\in \supp(M,\g)} \SFH(M,\g,\s);\]
this is a relatively $\Z_2$-graded abelian group. Furthermore, the relative grading
restricts to a relative grading on each term $\SFH(M,\g,\s)$.

It follows from work  of the third author~\cite{Ju06,Ju08}, see also Corollary~\ref{cor product},
that an irreducible balanced sutured manifold admits a taut depth~0 foliation
(or equivalently, it is diffeomorphic to a product $(R \times I, \partial R \times I)$ for
a compact oriented surface~$R$) if and only if $\SFH(M,\ga) \cong \Z$.
In particular, this shows that sutured Floer homology detects whether or not an irreducible balanced
sutured manifold admits a taut depth~0 foliation.

Our second main theorem states that sutured Floer homology also detects whether
or not an irreducible balanced sutured manifold $(M,\g)$ admits an indecomposable taut depth one foliation.
In order to state the theorem, we need to introduce a few more definitions.
First, we can associate to an indecomposable taut depth one foliation $\mF$ on~$(M,\ga)$
a cohomology class $\l(\mF)\in H^1(M)$ in a natural way, see Section~\ref{section:deffoliation} for details.
Second, given a cohomology class $\a \in H^1(M)$, we define
\[ \SFH_\a(M,\g):=\bigoplus \SFH(M,\g,\s),\]
where we take the direct sum over all $\spinc$-structures that pair minimally with~$\a$; we refer to the definition in equation (\ref{equ:sfha}) for details.
We can now formulate our second main theorem.

\begin{theorem} \label{mainthm}
Suppose $(M, \g)$ is a connected, irreducible balanced sutured manifold, and let $\a \in H^1(M)$. Then  $\SFH_\a(M,\g) \cong \Z$ if and only if there exists an indecomposable taut depth one foliation $\mF$ with $\l(\mF)=\a$.
\end{theorem}

The well-versed reader might be excused for having a sense of d\'ej\`a vu.
In fact, the theorem was proved by the first author~\cite{Al13} for strongly balanced sutured manifolds under the extra assumption
that $H_2(M)=0$.
This assumption was introduced for technical reasons to ensure that in the course of the proof one only works with balanced sutured manifolds.
We will resolve this technical issue using the Alexander polynomial of a sutured manifold
(see~\cite{FJR11} and~\cite{BP01}), and show that one can drop the condition on $H_2(M)$.
We also note that the proof of the `if' direction of this theorem is verbatim the same as in~\cite{Al13}.
We remove the strongly balanced condition on $(M,\g)$ with the aid of Theorem~\ref{thm nice}.
This includes a characterisation of the set of outer $\spinc$-structures for a nice taut decomposing surface
as the ones that pair minimally with~$[S]$. So in the decomposition formula for $\SFH$, we
do not have to refer to relative Chern classes, which are only defined when $(M,\g)$ is strongly balanced.

In Section~\ref{section:sfh}, we introduce the notion of an extremal $\spinc$-structure for a given balanced sutured manifold.
Loosely speaking, a $\spinc$-structure is extremal if it defines a vertex of the convex hull of the support of sutured Floer homology.
With this notion, we then immediately obtain the following corollary to Theorem~\ref{mainthm}.

\begin{corollary} \label{maincor}
A connected, taut balanced sutured manifold $(M,\g)$ admits an indecomposable taut depth one foliation
if and only if there exists an  extremal $\spinc$-structure with $\SFH(M,\g,\s) \cong \Z$.
\end{corollary}

The key new tool in the proof of Theorems \ref{mainthm-fib}
and \ref{mainthm} is Theorem~\ref{mainthmtechnical} that allows us to study fibrations and taut depth one foliations
even in the presence of non-trivial second homology. It states that, given a taut balanced sutured manifold~$(M,\g)$
and a class~$\a \in H^1(M)$ such that $\chi(\SFH_\a(M,\g)) \neq 0$,
if a properly embedded surface~$S$ is dual to~$\a$, then
the surface obtained by removing the closed components of~$S$ is also dual to~$\a$.
In particular, for every such~$\a \in H^1(M)$, there exists a decomposing
surface~$S$ dual to~$\a$ that \emph{has no closed components}, is well-groomed in the sense of Gabai \cite[p.~463]{Ga87},
and gives a taut decomposition. This ensures that the result~$(M',\g')$ of decomposing~$(M,\g)$ along~$S$ is
balanced (most importantly, $R_\pm$ has no closed components),
in which case we shall see in Theorem~\ref{thm nice} that $\SFH(M',\g') \cong \SFH_\a(M,\g)$. The proof of Theorem~\ref{mainthmtechnical}
relies on a careful study of the relationship of the surface~$S$ to an Alexander invariant of the pair~$(M,R_-)$ defined in~\cite{FJR11},
which is the Euler characteristic of sutured Floer homology.

The paper is organized as follows. In Section~\ref{section:sutures}, we recall the definition of a sutured manifold, the definition of a sutured manifold decomposition, and basic properties of sutured Floer homology.
In Section~\ref{section:proofprop}, we prove Theorem~\ref{mainthmtechnical}, which is our main technical result.
In Section~\ref{section:proofmainthm}
we first recall some basic properties and definitions for foliations on sutured manifolds
and we then  give a proof of Theorem~\ref{mainthm}.
Finally, in Section~\ref{section:fibred}, we state and prove our results that imply that sutured Floer homology
detects fibred classes on 3-manifolds with non-empty boundary.

\subsection*{Conventions and notations}
We assume that all $3$-manifolds are compact and oriented, and the foliations we refer to are smooth and transversely oriented.
Given a manifold~$W$ and a submanifold $S \subset W$, we denote by $N(S)$ a tubular neighbourhood of~$S$ in~$W$.
When referring to homology or cohomology groups with integer coefficients, we suppress the coefficient ring~$\Z$ in our notation.

\subsection*{Acknowledgment.} We are very grateful to the referee for many helpful comments.

\section{Sutured manifolds}\label{section:sutures}

In this Section we recall the definition of a sutured manifold, the definition of a sutured manifold decomposition, and basic properties of sutured Floer homology.

\subsection{Definition of sutured manifolds}
\label{section:defsutured}
A sutured manifold $(M,R_-,R_+,\gamma)$ consists of a $3$-manifold $M$,
 together with a decomposition of its boundary
\[ \partial M=-R_-\cup \g \cup R_+\]
into oriented submanifolds such that
\bn
\item $\g$ is a disjoint union of annuli,
\item $R_-$ and $R_+$ are disjoint,
\item if $A$ is a component of $\g$, then $R_-\cap A$ is connected and is
a boundary component of both~$A$ and~$R_-$, and similarly for $R_+ \cap A$. Furthermore,
\[
[R_+\cap A]=[R_-\cap A] \in H_1(A),
\]
where we endow $R_\pm \cap A$ with the orientation coming from the boundary of~$R_\pm$.
\en
Following the standard convention, we usually write $(M,\gamma)$ instead of
$(M,R_-,R_+,\gamma)$, and sometimes we will denote $R_\pm$ by $R_\pm(\g)$ when it is not clear from
the context which sutured manifold we are referring to.
Note that the notion of sutured manifolds is due to Gabai~ \cite{Ga83}, but our definition is
less general in so far as we do not allow ``toroidal sutures.''

Following Gabai, we say that a sutured manifold $(M,\g)$ is \emph{taut}
if $M$ is irreducible, and if $R_\pm$ are
incompressible and Thurston-norm minimising in $H_2(M,\g)$.
We are mostly interested in the study of taut foliations on sutured manifolds.
By the work of Gabai~\cite{Ga83}, if a sutured manifold carries a taut foliation, then it is taut.
It is therefore reasonable to restrict ourselves henceforth to the study of taut sutured manifolds.

We now recall the following definitions from \cite{Ju06} and \cite{Ju08}, respectively.
\bn
\item The sutured manifold $(M,\gamma)$ is  called \emph{balanced} if  $\chi(R_+) = \chi(R_-)$, and if~$M$ and~$R_\pm$  have no closed components.
\item
The sutured manifold $(M,\gamma)$ is said to be \emph{strongly balanced} if $(M,\g)$ is balanced, and if for every component $F$ of $\partial M$
the equality $\chi(F\cap R_-)=\chi(F\cap R_+)$ holds.
\en

An example of a strongly balanced, taut, sutured manifold is given as follows. Let~$R$ be
a compact oriented surface with no closed components.
Then
\[
(R \times [-1,1], \partial R \times [-1,1])
\]
forms a sutured manifold, which we refer to as a \emph{product sutured manifold}.
Note that $R_+ = R \times \{1\}$ and $R_- = R \times \{-1\}$.
It is clear that this sutured manifold is strongly balanced and taut.

If a sutured manifold $(M,\g)$ carries a codimension one, transversely oriented foliation~$\mF$,
then $\chi(R_+) = \chi(R_-)$, cf.~\cite[Proposition~3.6]{Ju10}. So if $(M,\g)$ carries a foliation,
has at least one suture on each boundary component, and $M$ has no closed components, then
it is balanced.

\subsection{Definition of sutured manifold decompositions}

We now define a key operation on a sutured manifold~$(M,\g)$, called a \emph{sutured manifold decomposition}.
Intuitively, this involves cutting the manifold $M$ along a properly embedded oriented surface~$S$, and
adding one side of~$S$ to $R_+$ and the other side to $R_-$. The new $\g$ is where the new $R_+$ and $R_-$ meet.

\begin{definition} \label{def decomp}
Let $(M,\ga)$ be a sutured manifold.
\bn
\item
We say that a properly embedded surface $S \subset M$ is a \emph{decomposing surface} if~$S$ has no closed components, and if given any component $A$ of $\ga$,  each component~$\l$
of~$S \cap A$
is either a properly embedded non-separating arc in $A$, or it is a simple closed curve  in the same homology class as $A\cap R_-$.
\item A decomposing surface $S$ for $(M,\ga)$ defines a {\it sutured manifold decomposition}
\[
(M,R_-,R_+,\ga) \leadsto^S (M',R_-',R_+',\ga'),
\]
where $M' := M \setminus \mbox{int}(N(S))$ and
\begin{align*}
\ga': & = (\ga \cap M') \cup N(S_+' \cap R_-) \cup N(S'_- \cap R_+),\\
R_+': & = ((R_+ \cap M') \cup S_+') \setminus \mbox{int}(\ga'), \\
R_-': & = ((R_- \cap M') \cup S_-') \setminus \mbox{int}(\ga'),
\end{align*}
where $S_+'$ (respectively $S_-'$) are the components of $\bdd N(S) \cap M'$ whose normal vectors points out of (respectively into) $M$.
\en
\end{definition}

If both $(M,\g)$ and $(M',\g')$ are taut, then $(M',\g')$ is simpler than $(M,\g)$ with respect
to some notion of complexity defined by Gabai~\cite{Ga83}.
Furthermore, as shown by Gabai~\cite{Ga83}, a sutured manifold $(M,\g)$ is taut if and
only if there is a sequence of sutured manifold decompositions starting with $(M,\g)$ and
ending with a product sutured manifold.

\subsection{Sutured Floer homology}\label{section:sfh}
As we mentioned in the introduction,
the third author~\cite{Ju06} introduced the notion of a relative $\spinc$-structure
on a balanced sutured manifold $(M,\g)$, and showed that the set
$\spinc(M,\g)$ of relative $\spinc$-structures on~$(M,\g)$ admits a canonical
free and transitive $H_1(M)$-action, which makes the set $\spinc(M,\ga)$ into a $H_1(M)$-torsor.

Furthermore, the third author~\cite{Ju06} associated to a $\spinc$-structure~$\s$  an abelian group $\SFH(M,\g,\s)$
called the {\it sutured Floer homology} of $(M,\ga)$ at $\fs$.
It extends the hat version of Heegaard Floer homology of closed 3-manifolds defined by
Ozsv\'ath and Szab\'o~\cite{OS04a,OS04b} to sutured 3-manifolds. The group
\[
\SFH(M,\ga):=\bigoplus \SFH(M,\ga,\fs)
\]
is finitely generated.
It thus follows in particular that the {\it support} of $\SFH(M,\ga)$,
which is defined to be
\[ \supp(M,\g):=\{ \s\in \spin(M,\g)\, \colon\, \SFH(M,\g,\s)\ne 0\},\]
is a finite set.
Furthermore, $\SFH(M,\g)$ carries a relative $\Z_2$-grading, which restricts to a $\Z_2$-grading on each summand $\SFH(M,\g)$.
Given an orientation of the vector space $H_*(M,R_-;\R)$, we obtain a lift of this relative
$\Z_2$-grading to an absolute one.

We pick an identification of $\spinc(M)$ with $H_1(M)$.
We say that $\al \in H^1(M;\R)$
{\it pairs minimally} with a $\spin^c$-structure
$\fs\in \spinc(M,\g)$ if $\al(\fs) \leq \al(\ft)$ for all $ \ft \in \supp(M,\g)$
and $\a(\s) = \a(\t)$ for some $\t \in \supp(M,\g)$.
We then define
\be \label{equ:sfha}
\SFH_\al(M,\ga):=\bigoplus_{\{\fs\in \spin^c(M,\ga) \colon \a \text{ pairs minimally with } \s\}}\SFH(M,\ga,\fs).
\ee
Note that we can and will view $\SFH_\al(M,\g)$ as an abelian group with a relative $\Z_2$-grading. Also note that this definition of $\SFH_\al(M,\ga)$ does not depend on the choice of identification of $\spin^c(M,\ga)$ with $H_1(M)$. Furthermore, this definition of  $\SFH_\al(M,\ga)$ is
clearly equivalent to  \cite[Definition~4.12]{Ju10} when $(M,\g)$ is strongly balanced.

Finally, we say that a $\spinc$-structure $\s\in \supp(M,\g)$ is \emph{extremal} if there exists an $\a\in H^1(M;\R)$ such that $\a$ pairs minimally with $\s$, and such that $\a$ does not pair minimally with any other $\t\in \supp(M,\g)$.

If $H_1(M)$ is torsion-free, then it is straightforward to see that $\s$ is extremal if and only if it is a vertex of the convex hull of $\supp(M,\g)$ viewed as a subset of $H_1(M;\R)=H_1(M)\otimes \R$.
Again, this  does not depend on the choice of the identification of $\spin^c(M,\ga)$ with $H_1(M)$.

\section{The main technical theorem}\label{section:proofprop}

In this section  we state and prove our main technical theorem.
This theorem is needed for the proofs of both of our main results, namely Theorems \ref{mainthm-fib} and \ref{mainthm}.

In order to state our main technical theorem  we need a few more definitions.
Let $(M,\ga)$ be a balanced sutured manifold.
\bn
\item
 We say that a surface $S \subset M$ {\it touches}~$R_-$ if $S \cap R_- \neq \emptyset$.
\item We say a properly embedded surface $S$ in $M$ and a cohomology class $\a \in H^1(M)$
are {\it dual} if~$\a$ is the Poincar\'e dual of $[S] \in H_2(M,\bdd M)$.
\item The group $\SFH(M,\g)$ carries a relative $\Z_2$-grading, which restricts to a grading on each $\SFH(M,\g,\s)$,
and thus also to each $\SFH_\a(M,\g)$. For a $\Z_2$-graded finitely generated abelian group~$G$,
we denote by~$\chi(G)$ the difference of the ranks of the two gradings. Note that~$\chi(G)$ is well-defined up to sign.
\en
 The goal of this section is to prove the following theorem.

\begin{theorem}\label{mainthmtechnical}
Let $(M,\g)$ be a connected  balanced sutured manifold,
and let $\a\in H^1(M)$ be a class such that
\[ \chi(\SFH_\a(M,\g))\ne 0.\]
If $\a$ is dual to a properly embedded surface $S$,
then $\a$ is also dual to the union of the components of $S$ that touch $R_-$.
\end{theorem}

The proof of Theorem~\ref{mainthmtechnical}
will occupy the remainder of this section, and is the technical heart of this paper. The key idea is to translate the information on the sutured Floer homology into information on a certain one-variable Alexander polynomial of $(M,R_-)$, and to then study the corresponding Alexander module of
$(M,R_-)$ and its relation to the surface $S$.

\subsection{Alexander polynomials of pairs of spaces}
Before we can introduce Alexander polynomials of sutured manifolds, we first introduce Alexander polynomials of pairs of spaces in general.

Let $F$ be a free abelian group, and let $M$ be a finitely generated $\Z[F]$-module. Since~$\Z[F]$ is Noetherian, there exists a finite resolution
\[ \Z[F]^r \xrightarrow{A} \Z[F]^s \to M \to 0.\]
By adding columns of zeros, we can assume that $r\geq s$. The \emph{order of $M$} is now defined as the greatest common divisor of all $s\times s$-minors of $A$. It is well-known that up to multiplication by a unit in $\Z[F]$, that is, up to multiplication by an element of the form $\pm f$ for $f\in F$, this definition does not depend on the choice of the resolution.
We  refer to \cite[Section~I.4.2]{Tu01} for details on orders.

Now let $(X,Y)$ be a pair of topological spaces.
Let $\psi \colon H_1(X) \to F$ be a homomorphism to a free abelian group.
We denote by $p \colon \ti{X} \to X$ the universal  abelian cover of $X$, and we write $\ti{Y}:=p^{-1}(Y)$. Note that
$C_*(\ti{X},\ti{Y};\Z)$ is a $\Z[H_1(X)]$-module. We can thus consider the following chain complex of $\Z[F]$-modules:
\[ C_*^\psi(X,Y;\Z[F]):=C_*(\ti{X},\ti{Y};\Z)\otimes_{\Z[H_1(X)]}\Z[F].\]
We then denote  the corresponding homology modules by
\[ H_*^\psi(X,Y;\Z[F]).\]
(We will drop the superscript $\psi$ if it is understood from the context.)
Furthermore, we  denote by
\[ \Delta_{X,Y}^\psi\in \Z[F]\]
the order of the $\Z[F]$-module $H_*^\psi(X,Y;\Z[F])$.
Note that $\Delta_{X,Y}^\psi$ is well-defined up to multiplication by a unit in $\Z[F]$, that is, up to multiplication by an element of $\Z[F]$ of the form $\pm f$ for $f\in F$.
Several times, we will be interested in homomorphisms of $H_1(X)$ to $F = \Z$. In this case, we identify the group ring of $F$ with $\zx$.

We denote by $\Q(F)$ the quotient field of $\Z[F]$, and define
$H_*^\psi(X,Y;\Q(F))$ to be the homology of the chain complex
\[ C_*^\psi(X,Y;\Q(F)):=C_*(\ti{X},\ti{Y};\Z)\otimes_{\Z[H_1(X)]}\Q(F).\]
The following  elementary lemma is an immediate consequence of the fact that $\Q(F)$
is flat over $\Z[F]$.

\begin{lemma}\label{lem:equiv}
The Alexander polynomial $\Delta_{X,Y}^\psi$ is non-zero
if and only if
\[
H_*^\psi(X,Y;\Q(F))=0.
\]
\end{lemma}

We also have  the following basic lemma.

\begin{lemma}\label{lem:delta0}\label{lem:h0}
Let $(X,Y)$ be a pair of compact topological spaces with $X$ connected.
Let $\psi\colon  H_1(X) \to F$ be a homomorphism to a free abelian group.
Then
\[ H_0^\psi(X,Y;\Q(F)) \cong \left\{ \ba{ll}
\Q(F)&\text{if } Y=\emptyset \text{ and $\psi$ is trivial,}\\
0& \text{otherwise.}
\ea \right.\]
\end{lemma}

\begin{proof}
We first consider the case when $Y$ is the empty set.
Since $X$ is connected, it  follows from the standard expression for the 0-th twisted homology (see  \cite[Section~VI]{HS97}) that we have  an isomorphism
\[ \xymatrix{ H_0^\psi(X;\Q(F))\ar[r]^-\cong &\Q(F)/\{ v-\psi(g)v\,\colon\, v\in \Q(F)\mbox{ and }g\in H_1(X)\}.}\]
Since $\Q(F)$ is a one-dimensional vector space, it follows that $H_0^\psi(X;\Q(F))=0$ if
we quotient out by a non-zero subspace. But this  in turn happens if and only if  $\psi \colon  H_1(X)\to F$ is non-trivial.

Now suppose that $Y$ is non-empty. We denote the components of $Y$ by $Y_i$ for $i \in I$.
Let $\phi_i \colon H_1(Y_i; \Z) \to \F$ be the composition of the inclusion map $H_1(Y_i) \to H_1(X)$ with~$\psi$,
and we define $\phi \colon H_1(Y; \Z) \to F$ analogously.
We then get  a commutative diagram
\[ \xymatrix{ H_0^\phi(Y;\Q(F))\ar[d]\ar[r]^-\cong &\bigoplus\limits_{i\in I}\Q(F)/\{ v-\phi_i(g)v\, \colon \, v\in \Q(F)\mbox{ and }g\in H_1(Y_i)\}\ar[d]\\
H_0^\psi(X;\Q(F))\ar[r]^-\cong &\Q(F)/\{ v-\psi(g)v\, \colon \, v\in \Q(F)\mbox{ and }g\in H_1(X)\},}\]
where the left vertical map is  the map induced by the inclusion, and the right vertical map is the sum of the canonical maps. Since each map
$\phi_i \colon  H_1(Y_i) \to F$ factors through $\psi \colon  H_1(X)\to F$,
it follows that for each $i \in I$ we have
\[ \im(\phi_i \colon  H_1(Y_i) \to F) \subset \im(\psi \colon  H_1(X) \to F),\]
which in turn implies that the vertical map  on the right in the above commutative diagram is an epimorphism for each summand.
Since we have at least one summand, the vertical map on the right is an epimorphism.
Now it follows  from the short exact sequence
\[ H_0^\phi(Y; \Q(F)) \to H_0^\psi(X; \Q(F)) \to H_0^\psi(X,Y; \Q(F)) \to 0\]
that $H_0^\psi(X,Y;\Q(F))=0$.
\end{proof}

\subsection{The Alexander polynomial  of a sutured manifold}

Given a sutured manifold $(M,\g)$ and a homomorphism $\psi \colon H_1(M) \to F$, where
$F$ is a free abelian group, we refer to $\Delta_{M,R_-}^\psi$ as \emph{the Alexander polynomial of the triple $(M,\g,\psi)$}. We will now relate the Alexander polynomials of $(M,\g)$ to the sutured Floer homology of $(M,R_-)$.
In order to do this, we need one more piece of notation: if $H$ is an abelian group, then given
$p$, $q\in \Z[H]$, we write $p \doteq q$ if $p$ and $q$ are equal  up to multiplication by an
element of $\Z[H]$ of the form $\pm h$ for $h \in H$.

Recall that,  given a  balanced sutured manifold $(M,\g)$ and an orientation of the vector space $H_*(M,R_-; \R)$,
the group $\SFH(M,\g)$ is an absolutely $\Z_2$-graded, finitely generated abelian group.
Since this $\Z_2$-grading restricts to a $\Z_2$-grading on each summand $\SFH(M,\g,\s)$,
it thus makes sense to consider the Euler characteristic $\chi(\SFH(M,\g,\s)) \in \Z$ for any $\spinc$-structure
$\s \in \spinc(M,\g)$.
Using the results and techniques of~\cite{FJR11}, we will now prove the following proposition.

\begin{proposition}\label{prop:sfhdelta}
Let $(M,\g)$ be a connected balanced sutured manifold. Pick an orientation of $H_*(M,R_-; \R)$
and an identification of $\spinc(M,\g)$ with $H = H_1(M)$. Let $\psi \colon H \to F$ be a
homomorphism to a free abelian group. Then we have
\[
\Delta_{M,R_-}^\psi \doteq \sum_{\s \in \spinc(M,\g)} \chi(\SFH(M,\g,\s)) \cdot \psi(\s)\in \Z[F].
\]
\end{proposition}

\begin{proof}
First of all, given a finitely generated abelian group $H$, we denote by $\Q(H)$ the
ring given by inverting all elements in $\Z[H]$ that are not zero divisors.
Given  a pair of finite CW-complexes $(X,Y)$, we then denote by
$\tau(X,Y) \in \Q(H_1(X))$ Turaev's maximal abelian torsion of $(X,Y)$.
We refer to \cite[Section~3.4]{FJR11} and \cite{Tu01,Tu02} for details.

We equip the pair $(M,R_-)$ with a CW-structure $(X',Y')$, and identify $H_1(X')$ with~$H$.
By~\cite[Theorem~1.1]{FJR11}, we know  that $\tau(X',Y')$ lies  in $\Z[H] \subset \Q(H)$,
and that in fact
\[ \tau(X',Y') \doteq \sum_{\s\in \spinc(M,\g)} \chi(\SFH(M,\g,\s)) \cdot \s \in \Z[H].\]
Since $(M,\g)$ is balanced, $R_-\ne \emptyset$.
It now follows from standard arguments that the pair of CW-complexes $(X',Y')$ is simple homotopy equivalent to a pair of CW-complexes $(X,Y)$, where $X$ is a 2-complex  such that all
0-simplices of $X$ lie in $Y$.
Since Reidemeister torsion is invariant under simple homotopy equivalence, it follows that
\[ \tau(X,Y) \doteq \tau(X',Y') \in \Z[H],\]
where we also identify $H_1(X)$ with $H$.

Since $(M,\g)$ is a balanced sutured manifold,
it follows from a standard Poincar\'e duality argument that $\chi(M) = \chi(R_-)$.
Hence
\[
\chi(X,Y) = \chi(M,R_-) = \chi(M)-\chi(R_-)=0.
\]
We denote by $r$ the number of $1$-cells in $X\sm Y$.
Since $X$ is a  2-complex, and since all 0-cells of $X$ lie in $Y$, it follows from
$\chi(X,Y)=0$ that $r$ also equals  the number of 2-cells in $X \sm Y$.

Let $p\colon \ti{X}\to X$ be the universal
abelian cover of $X$, and we write $\ti{Y}:=p^{-1}(Y)$.
By picking lifts of the cells in $X\sm Y$ to $\ti{X}$, we can identify
$C_1(\ti{X},\ti{Y})$ with $\Z[H]^r$ and
$C_2(\ti{X},\ti{Y})$ with $\Z[H]^r$. We denote by $A$ the matrix over $\Z[H]$ representing the boundary map
$C_2(\ti{X},\ti{Y}) \to C_1(\ti{X},\ti{Y})$ with respect to the chosen bases. It then follows from \cite[Lemma~3.6]{FJR11} that
\[ \tau(X,Y) \doteq \det(A)\in \Z[H].\]

We also denote by $\psi$ the extension of $\psi \colon H \to F$
to a ring homomorphism $\Z[H]\to \Z[F]$.
Since the Alexander polynomial of a pair only depends on the homotopy type of the pair, we see that $\Delta_{M,R_-}^\psi=\Delta_{X,Y}^\psi$.
Furthermore, note that
\[  C_2(\ti{X},\ti{Y})\otimes_{\Z[H]}\Z[F]\xrightarrow{\partial_2} C_1(\ti{X},\ti{Y})\otimes_{\Z[H]}\Z[F]\to H_1^\psi(X,Y;\Z[F])\to 0\]
is a free resolution for $H_1^\psi(X,Y;\Z[F])$. This resolution is isomorphic to the resolution
\[ \Z[F]^r\xrightarrow{\psi(A)}\Z[F]^r\to H_1^\psi(X,Y;\Z[F])\to 0.\]
It  follows from the definition of the order of a $\Z[F]$-module that $\Delta_{X,Y}^\psi\doteq \det(\psi(A))$,
where $\psi(A)$ is the matrix given by applying $\psi\colon \Z[H]\to \Z[F]$ to all entries of $A$.
Combining the above equalities we then see that
\[  \ba{rcl}\Delta_{M,R_-}^\psi\doteq \Delta_{X,Y}^\psi \doteq \det(\psi(A))&=&\psi(\det(A))\\
&\doteq &\psi(\tau(X,Y))\\&\doteq&\psi(\tau(X',Y'))\\
&\doteq &\psi\left(\sum_{\s\in \spinc(M,\g)}\chi(\SFH(M,\g,\s))\cdot \s\right)\\
&\doteq &\sum_{\s\in \spinc(M,\g)}\chi(\SFH(M,\g,\s))\cdot \psi(\s).\ea \]
\end{proof}

\subsection{Proof of Theorem \ref{mainthmtechnical}}
In this  section, we prove  the following proposition.

\begin{proposition}\label{mainprop}
Let $(M,\g)$ be a connected balanced sutured manifold, let
\[
\a \in H^1(M) \cong \hom(H_1(M),\Z),
\]
and  suppose that $S$ is a properly embedded surface in $M$ that is Poincar\'e dual to~$\a$.
If $\De^\a_{M,R_-} \neq 0$, then the class $\a$ is also dual to the union of the components of $S$ that touch $R_-$.
\end{proposition}

Before we give the proof of  Proposition \ref{mainprop}, we will first
show that Theorem \ref{mainthmtechnical} follows from Proposition \ref{mainprop}.

\begin{proof}[Proof of Theorem \ref{mainthmtechnical}]
Let $(M,\g)$ be a connected  balanced sutured manifold, and suppose that
 $\a\in H^1(M)$ is a class such that
\[ \chi(\SFH_\a(M,\g))\ne 0.\]
Furthermore, let $S$ be a properly embedded surface dual to $\a$.
We have to show that~$\a$ is also dual to the union of the components of $S$ that touch $R_-$.

First, we pick an identification of $\spinc(M,\g)$ with $H_1(M)$. By Proposition~\ref{prop:sfhdelta},
\[
\Delta_{M,R_-}^\a\doteq \sum_{\s\in \spinc(M,\g)}\chi(\SFH(M,\g,\s))\cdot x^{\a(\s)}\in\zx.
\]
We write
\[
d:=\min \{ \a(\s)\,\colon\, \s\in \supp(M,\g)\}.
\]
It follows from the definition of $\SFH_\a(M,\g)$
that $\Delta_{M,R_-}^\a$  is of the form
\[
\Delta_{M,R_-}^\a = \chi(\SFH_\a(M,\g))\cdot x^{d}+ \mbox{higher order terms}.
\]
Hence, the assumption $ \chi(\SFH_\a(M,\g))\ne 0$ implies that $\Delta_{M,R_-}^\a$ is non-zero.
The result now follows from Proposition~\ref{mainprop}.
\end{proof}

We  now turn to the proof of Proposition~\ref{mainprop}.

\begin{proof}[Proof of Proposition~\ref{mainprop}]

We denote the components of $S$ by $S_e$ for $e\in E$.
We pick disjoint open tubular neighbourhoods $S_e \times [-1,1]$ for $e \in E$
such that
\begin{itemize}
\item $S_e\times \{0\} = S_e$,
\item  $S_e\times [-1,1]$ has the same orientation as $M$, and
\item $(S_e\times [-1,1]) \cap \partial M = \partial S_e \times [-1,1]$.
\end{itemize}
We denote by  $W_v$ for $v\in V$ the components of $M \sm \bigcup_{e\in E} \left(S_e \times (-1,1)\right)$.
Given an element $e\in E$,
we denote by $i(e)$ the element in $V$ with $S_e \times \{-1\} \subset W_{i(e)}$,
and similarly,  we denote by $t(e)$ the element in $V$ with $S_e \times \{1\} \subset W_{t(e)}$.
The pair $(V,E)$, together with the maps $i \colon E\to V$ and $t \colon E \to V$,
defines a directed graph that we denote by $\Gamma$.
We equip $\Gamma$ with the usual topology. Note that  vertices are closed sets, and we view edges without the boundary points, i.e. we view edges as open sets.

We now refer to a vertex or an edge of $\Gamma$ as ``green'' if the corresponding subset of $M$ touches $R_-$, otherwise we refer to a vertex or an edge as ``black.''
We denote by $G$ the union of the green vertices and edges, and we denote by $B$ the union of the black vertices and edges.
Note that  the fact that  $R_- \ne \emptyset$ implies that there exists at least one vertex that is green.
If the whole graph is green, then there is nothing to prove. So assume that this is not the case.

Note that if a component $W_v$ does not touch $R_-$, then neither does any of its boundary components. In other words, if a vertex of $\Gamma$ is black, then all the edges adjacent
to it are also black. In particular, $B$ is an open subset of $\Gamma$.
In the next lemma, we will see that $B$ is also ``almost'' closed.

In order to state the next lemma, we need the notion of a path in $\Gamma$.
We view any interval $[0,n]$ with $n \in \N$ as a graph in the canonical way with $n+1$ vertices and $n$ edges.
We define a \emph{path} in $\Gamma$ to be a continuous
map $P \colon [0,n] \to \Gamma$ that sends vertices to vertices and edges to edges,
and such that $P$ restricted to $(0,n)$ is injective.
We can now formulate the following lemma.

\begin{lemma} \label{lem:b}
If $P\colon [0,n] \to \Gamma$ is a path in $\Gamma$ such that $P((0,n))$ is black
and $P(0)$ is green, then $P(n)$ is black.
\end{lemma}

We postpone the proof of the lemma, and first show how to deduce Proposition~\ref{mainprop}
from Lemma~\ref{lem:b}.

Recall that before the lemma we showed that $B$ is open.
Since $\Gamma$ is connected and since $B$ is not all of $\Gamma$, it follows that
$B$ is not closed.
This means that there exists a black edge $f$ with at least one green boundary vertex.
Without loss of generality, we can assume that $i(f)$ is green.
Assuming Lemma~\ref{lem:b}, we can now prove the following claim.

\begin{claim}
The edge $f$ is separating.
\end{claim}

In order to show that $f$ is separating, it suffices to prove that for
each path $P \colon [0,k] \to \Gamma$ with  $P(0) = i(f)$ and $P((0,1)) = f$,
we have $P(k) \ne P(0)$.
So let~$P$ be such a path.
Then there exists a maximal $n \in \{\,1,\dots, k \,\}$ such that $P((0,n)) \subset B$.
It follows from Lemma~\ref{lem:b} that $P(n)$ is also black. By the openness of~$B$,
we furthermore know that $P((0,n+1) \cap (0,k])$ is also black.
It  follows from the maximality of $k$ that $k = n$. We thus see that $P(k)$ is black, whereas $P(0) = i(f)$ is green; i.e., we have shown that $P(k) \ne P(0)$.
This concludes the proof of the claim.

We denote by $\Gamma'$ the component of $\Gamma \sm f$ that contains the vertex $i(f)$. We then denote by $X$ the union of the  vertex and edge spaces that correspond to $\Gamma'$, that is
\[ X: = \bigcup\limits_{v \in V(\Gamma')} W_v \cup \bigcup \limits_{e\in E(\Gamma')} S_e \times [-1,1].\]
It then follows easily from the definitions that
\[ \partial X=S_f\times \{-1\} \,\,\cup \,\,(\partial X\cap \partial M).\]
Put differently, $S_f$ represents the trivial element in $H_2(M,\partial M)$. We thus see that  the surfaces $S$ and $\cup_{e\ne f} S_e$
represent the same homology class in $H_2(M,\partial M)$.
Since  $S_f$ does not  touch  $R_-$, the proposition now follows by induction on the number of components of $S$.

We now turn to the proof of Lemma \ref{lem:b}.

\begin{proof}[Proof of Lemma \ref{lem:b}]
In the following, given subsets $X\subset M$ and $Y\subset X$,
we denote by $H_*(X,Y;\Q(x))$ the twisted homology corresponding to the homomorphism
$H_1(X) \to H_1(M) \xrightarrow{\a} \Z$.
Furthermore, given an edge $e\in E$, we denote the map
\[
\ba{rcl}
H_0(S_e, S_e\cap R_-;\Q(x))& \to&H_0(W_{i(e)},W_{i(e)}\cap R_-;\Q(x))\\[0mm]
[c] &\mapsto & [c\times \{-1\}] \ea
\]
by $i_e$. Similarly, we denote the map
\[
\ba{rcl} H_0(S_e,S_e\cap R_-;\Q(x))& \to&H_0(W_{t(e)},W_{t(e)}\cap R_-;\Q(x))\\[0mm]
[c] &\mapsto & [c\times \{1\}]\ea
\]
by $t_e$.
These maps now give rise to
the following Mayer--Vietoris sequence:
\[
\ba{ccccccccc}
\hspace{-0.15cm}\hspace{-0.15cm}&\hspace{-0.15cm}\hspace{-0.15cm}&\hspace{-0.15cm}\cdots\hspace{-0.15cm}&\hspace{-0.15cm}
\to\hspace{-0.15cm}&\hspace{-0.15cm} H_1(M,R_-;\Q(x))\hspace{-0.15cm}&\hspace{-0.15cm}\to &\\
\hspace{-0.15cm}\bigoplus\limits_{e\in E} H_0(S_e,S_e\cap R_-;\Q(x)) \hspace{-0.15cm}&\hspace{-0.15cm}\xrightarrow{\bigoplus\limits_{e\in E}i_e-t_e} \hspace{-0.15cm}&\hspace{-0.15cm}\bigoplus\limits_{v\in V} H_0(W_v,W_v\cap R_-;\Q(x)) \hspace{-0.15cm}&\hspace{-0.15cm}\to \hspace{-0.15cm}&
\hspace{-0.15cm}H_0(M,R_-;\Q(x))\hspace{-0.15cm}&\hspace{-0.15cm}\to&\hspace{-0.15cm}\ea \]
Since $(M,\ga)$ is a balanced sutured manifold, we have $R_- \neq \emptyset$. So by Lemma~\ref{lem:h0}, we have that $H_0(M,R_-;\Q(x))=0$.
Next, since $\De_{M,R_-}^\a \neq 0$, Lemma~\ref{lem:equiv} tells us that $H_1(M,R_-;\Q(x))=0$.
Therefore the map
\[
 \bigoplus\limits_{e\in E} H_0(S_e,S_e\cap R_-;\Q(x)) \xrightarrow{\bigoplus\limits_{e\in E}i_e-t_e} \bigoplus\limits_{v\in V} H_0(W_v,W_v\cap R_-;\Q(x))
 \]
in the above long exact sequence  is in fact an isomorphism.
Note that $\a$ restricted to each $W_v$ is zero, since $S$ is Poincar\'e dual to $\a$.
Furthermore,  any curve $c$ in $S_e$ can be pushed off~$S$,
hence the restriction of  $\a$  to any $S_e$ is also zero.
It thus follows from  Lemma \ref{lem:h0} that for  $e\in E$ we have
\be \label{equ:scolor} H_0(S_e,S_e\cap R_-;\Q(x)) \cong \left\{ \ba{ll}
0 &\text{if $e$ is green,}\\
\Q(x)& \text{if $e$ is black.}
\ea \right. \ee
 Similarly, we see that for $v \in V$, we have
 \be \label{equ:wcolor}  H_0(W_v,W_v \cap R_-; \Q(x)) \cong \left\{ \ba{ll}
0 &\text{if $v$ is green,}\\
\Q(x) & \text{if $v$ is black.}
\ea \right. \ee

Now let $P\colon [0,n]\to G$ be a path in $G$ such that $P((0,n))$ is black and $P(0)$ is green.
We then consider the diagram
\[
\xymatrix{
\bigoplus\limits_{\substack{e \text{ edge} \\ \text{of } [0,n]}} 
H_0(S_{P(e)},S_{P(e)}\cap R_-;\Q(x))
\ar[rr]^{\bigoplus i_{P(e)}-t_{P(e)}}
\ar[dd]&& \bigoplus\limits_{\substack{v \text{ vertex}\\ \text{of } [0,n]}}   
H_0(W_{P(v)},W_{P(v)}\cap R_-;\Q(x))\ar[dd]\\  \\
\bigoplus\limits_{e\in E} H_0(S_e,S_e\cap R_-;\Q(x))
\ar[rr]^{\bigoplus i_{e}-t_{e}}_{\cong}
&& \bigoplus\limits_{v\in V} H_0(W_v,W_v\cap R_-;\Q(x)),}
\]
\noindent where the vertical maps are for each summand just given by the canonical inclusion.
Now recall that the restriction of $P$ to $(0,n)$ is injective, it thus follows that the left vertical map in the above commutative diagram is a  monomorphism.
Furthermore, note that this diagram is in fact commutative by the choice of our horizontal maps.

The hypothesis that the image of each edge in $P$ is black implies by~\eqref{equ:scolor} that  in the above commutative diagram, the top-left $\Q(x)$-module
is isomorphic to $\Q(x)^n$.  Furthermore, the hypothesis that $P(0)$ is green and that $P((0,n))$ is black implies by~\eqref{equ:wcolor} that
the top-right $\Q(x)$-module of the above commutative diagram is isomorphic to $\Q(x)^{n-1}\oplus
H_0(W_{P(n)},W_{P(n)}\cap R_-;\Q(x))$. We can thus rewrite the above commutative diagram as
\[
\xymatrix{
 \Q(x)^n \ar[rr]\ar[d]&& \Q(x)^{n-1}\oplus H_0(W_{P(n)},W_{P(n)}\cap R_-;\Q(x))\ar[d]\\
 \bigoplus\limits_{e\in E} H_0(S_e,S_e\cap R_-;\Q(x)) \ar[rr]^\cong&& \bigoplus\limits_{v\in V} H_0(W_v,W_v\cap R_-;\Q(x)).}
 \]
Since the left vertical map is a monomorphism, and since the bottom horizontal map is an isomorphism, it follows from the commutativity that the top horizontal map is a monomorphism. But this is only possible if $H_0(W_{P(n)},W_{P(n)}\cap R_-;\Q(x)) \ne 0$, which means that $P(n)$ is black.
\end{proof}

\end{proof}

\section{The proof of Theorem~\ref{mainthm}} \label{section:proofmainthm}

The proof of Theorem~\ref{mainthm} is
closely modelled on the  proof of~\cite[Theorem~A]{Al13}.
Proposition~\ref{mainprop} is the key new ingredient in the proof
that  allows us to deal with the case when $H_2(M) \ne 0$.

\subsection{Foliations on sutured manifolds}
\label{section:deffoliation}
In this section, we recall several basic definitions and results
about foliations that we need in the statement of Theorem~\ref{mainthm}.

Let $(M,\g)$ be a  sutured manifold.
As we explained in the introduction,  a \emph{foliation on a sutured manifold $(M,\g)$} is a foliation on the 3-manifold $M$ such that all the leaves of $\mF$ are transverse to $\ga$ and tangential to $R_-$ and $R_+$.
Furthermore,  a foliation $\mF$ is \emph{taut} if  there exists a curve or properly embedded arc in $M$ that is transverse to the leaves of $\mF$ and that intersects every leaf of $\mF$ at least once.

\begin{definition} \label{def:holonomy}
Let~$L$ be a leaf of the foliation~$\mF$ on the manifold~$M$. Given a basepoint~$x \in L$ and a transversal~$\tau$ at~$x$,
the holonomy of the leaf~$L$ at~$x$ is a homomorphism
\[
H \colon \pi_1(L,x) \to \text{Homeo}(\tau),
\]
where $\text{Homeo}(\tau)$ is the group of germs of homeomorphisms of~$\tau$. Given a loop~$\g$ based at~$x$, choose a covering of~$\g$
by foliation charts $U_0, \dots, U_k$, labeled such that there is a subdivision $0 = t_0 < t_1 < \dots < t_k = 1$ of $[0,1]$
with $\g|_{[t_i,t_{i+1}]}$ lying in~$U_i$ for $i \in \{0,\dots,k-1\}$. Let~$\tau_0 = \tau \cap U_0 \cap U_k$, and pick a transversal~$\tau_i$
at~$\g(t_i)$ for~$i \in \{1,\dots, k-1\}$.
Using the product structures of the foliation charts, we obtain a sequence of homeomorphisms $h_0 \colon \tau_0 \to \tau_1$,
$h_1 \colon \tau_1 \to \tau_2$, \dots, $h_{k+1} \colon \tau_k \to \tau_0$ (where we map $p \in \tau_i$ to the intersection of the plaque in~$U_i$
through~$p$ with~$\tau_{i+1}$), and composing the germs of these we obtain
a germ of a homeomorphism $\tau \to \tau$ independent of the choice of foliation charts~$U_i$, transversals~$\tau_1, \dots, \tau_k$,
and the homotopy class of~$\g$; we define~$H([\g])$ to be this germ.
\end{definition}

We now recall from the introduction that a leaf $L$ of a foliation $\mF$ on $(M,\g)$ is said to be {\it depth~0} if it is compact. Recursively, we now define a leaf to be of {\it depth~$k$}
if it is not of depth less than $k$,  and if $\overline L \setminus L$ is a collection of leaves of depth less than $k$. Otherwise, we say that $L$ is of depth~$\infty$.
A depth $k$ foliation $\mF$ of $(M,\ga)$ is a foliation such that all leaves have depth at most $k$, and $\mF$ contains at least one leaf of depth~$k$.

If $\mF$ is a depth zero foliation of the sutured manifold~$(M,\g)$ without closed components,
then the Reeb stability theorem~\cite[Lemma~6, p73]{CN85} implies that~$(M,\g)$ is a product sutured manifold $(R \times [-1,1], \partial R \times [-1,1])$,
and $\mF$ is equivalent to the product foliation whose leaves are the components of $R \times \{t\}$ for $t \in [-1,1]$.
Indeed, if every leaf of~$\mF$ is compact, they must all have trivial holonomy, and then by Reeb stability the foliation
is a product in the neighbourhood of each leaf; i.e., a fibration. A fibration of a sutured manifold without closed components by compact leaves must be the product fibration as the leaf space is a disjoint union of finitely many copies of~$[0,1]$.

We will now study taut depth one foliations a little further.
As we mentioned in the introduction, in this paper, we say that a foliation of depth greater than 0 is  \emph{indecomposable} if the components of $R_\pm$ are the only compact leaves.
If $\mF$ is an indecomposable taut depth one foliation on $(M,\ga)$,
then it follows from Reeb stability (cf.\ \cite[Lemma~11.4.4]{Foli2}) that there exists a fibre
bundle  $p\colon M \setminus (R_- \cup R_+) \to S^1$
such that the fibres are precisely the leaves of $\mF$. We denote by
\[ \la(\mF)\colon \pi_1(M)\xrightarrow{\cong}\pi_1(M\sm (R_-\cup R_+)) \xrightarrow{p_*} \Z\]
the induced map, and we denote by $\la(\mF) \in H^1(M) \cong \hom(\pi_1(M),\Z)$ also the
corresponding  cohomology class.

\subsection{Summary of results on sutured Floer homology}

Given a decomposing surface $S$ in a sutured manifold,
Gabai~\cite[Definition~0.2]{Ga87} and the third author~\cite[Definition~3.20]{Ju10} introduced the notions of $S$ being \emph{well-groomed} and  \emph{nice}, respectively.

\begin{definition} \label{def:groomed}
Let~$(M,\g)$ be a sutured manifold, together with a decomposing surface~$S$, and let~$(M',\g')$
be the result of decomposing~$(M,\g)$ along~$S$.
Then~$S$ is \emph{well-groomed} if both~$(M,\g)$ and~$(M',\g')$ are taut,
and for each component~$V$ of~$R(\g)$ the intersection~$S \cap V$
is a union of parallel, coherently oriented, non-separating closed curves,
or parallel, coherently oriented arcs joining two different components of~$\partial V$.

The decomposing surface~$S$ is \emph{nice} if it has no closed components, $\partial S$ is transverse to the horizontal
foliation of~$\g$ by circles, and for each component~$V$ of~$R(\g)$, the set of closed components of~$S \cap V$
consists of parallel, coherently oriented, and boundary-coherent simple closed curves. Here, an oriented
simple closed curve~$C$ in~$V$ is called
\emph{boundary coherent} if either~$[C] \neq 0$ in $H_1(V)$, or if~$[C] = 0$ and
it is oriented as the boundary of its interior.
\end{definition}

\begin{lemma} \label{lem:groomedandnice}
Let $(M,\ga) \leadsto^S (M',\g')$ be a sutured manifold decomposition. Then the following hold:
\bn
\item \label{it:1} If~$S$ is well-groomed, then $(M,\g)$ and $(M',\g')$ are both taut.
\item \label{it:2} If $S$ is nice and $(M,\g)$ is balanced, then $(M',\g')$ is also balanced.
\item \label{it:3} If $(M,\ga)$ is taut, then any non-zero class in $H^1(M)$ is Poincar\'e dual to a properly embedded well-groomed decomposing surface.
\item \label{it:4} Any union of components of a well-groomed surface is again well-groomed.
\item \label{it:5} A well-groomed decomposing surface \emph{with no closed components} can be perturbed into a nice one.
\en
\end{lemma}

\begin{proof}
Statement~\eqref{it:1} is part of the definition of well-groomed surfaces.
For~\eqref{it:2}, first observe that $\chi(R_\pm(\g')) = \chi(R_\pm(\g)) + \chi(S)$, so
$\chi(R_+(\g)) = \chi(R_-(\g))$ implies~$\chi(R_+(\g')) = \chi(R_-(\g'))$. If~$M$ has no closed components,
neither will~$M'$. Finally, we check that each component of~$\partial M'$ has at least one suture.
As~$S$ has no closed components, and since for each component~$V$ of~$R(\g)$, the set of closed components
of~$S \cap V$ consists of parallel, coherently oriented curves, an issue could only arise if these
curves were null-homologous in~$V$. Since the innermost curve~$C$ is oriented as~$\partial S$
and also as the boundary of its interior~$F \subset V$,
the component~$S_0$ of~$S_\pm$ neighboring~$F$ in~$\partial M'$ induces the same orientation on~$\partial F$ as~$F$,
and hence~$S_0$ and~$F$ are oriented oppositely in~$R(\g')$. In particular, $N(\partial F) \subset \g'$,
so there is a suture in the component of~$\partial M'$ containing~$F$. A similar argument applies for
every component of~$\partial M'$ containing an annulus between two parallel closed components of~$S \cap V$.

The third statement is precisely~\cite[Lemma~0.7]{Ga87}.
The fourth is a straightforward consequence of the definitions,
together with~\cite[Lemma~3.5]{Ga83}, which states that if~$(M,\g) \rightsquigarrow^S (M',\g')$ is a sutured manifold
decomposition and $(M',\g')$ is taut, then so is~$(M,\g)$ (assuming~$\g$ has no toroidal components).
Hence, if~$S$ is well-groomed, and if~$S_1$ is a union of some components of~$S$, then let~$S_2 = S \setminus S_1$.
Consider the decompositions
\[
(M,\g) \rightsquigarrow^{S_1} (M_1,\g_1) \rightsquigarrow^{S_2} (M',\g').
\]
Note that~$\g_1$ has no toroidal components since~$S_1$ satisfies the boundary conditions on well-groomed surfaces.
As~$(M',\g')$ is taut, we can apply~\cite[Lemma~3.5]{Ga83} to obtain that~$(M_1,\g_1)$ is also taut.
Hence~$S_1$ is indeed well-groomed.
Finally, \eqref{it:5} is~\cite[Remark~3.21]{Ju10}  (any  groomed surface with no closed components can be made into a
nice surface by a small perturbation that places its boundary into general position).
\end{proof}

\begin{definition}
Let $S$ be a decomposing surface in a balanced sutured manifold $(M,\g)$.
Following \cite[Definition~1.1]{Ju08} we say that $\s \in \spinc(M,\g)$ is \emph{outer with respect to~$S$} if there is a unit vector field~$v$
on~$M$ whose homology class is~$\s$, and such that $v_p \neq -(\nu_S)_p$ for every $p \in S$. Here $\nu_S$ is the
unit normal vector field of~$S$ with respect to some  Riemannian metric on~$M$.
We denote by $O_S \subset \spinc(M,\g)$ the set of all
outer $\spinc$-structures with respect to $S$.
\end{definition}

\begin{lemma} \label{lem:min}
Let $(M,\ga)$ be a taut balanced sutured manifold, and let $\al \in H^1(M)$.
Suppose that~$S$ is a nice decomposing surface that is dual to~$\al$.
If the result of the sutured manifold decomposition $(M,\ga) \leadsto^S (M',\ga')$ is taut, then
$\s \in \spinc(M,\g)$ pairs minimally with~$\a$ if and only if $\s \in O_S$.
\end{lemma}

\begin{proof}
Let $\s \in O_S$. First, we show that for $\s' \in \spinc(M,\g)$, we have
$\langle\, \s -\s', \a \,\rangle = 0$ if and only if~$\s' \in O_S$.
Pick unit vector fields~$v$ and~$v'$ representing~$\s$ and~$\s'$, respectively,
such that~$v$ does not coincide with $-\nu_S$ at any point of~$S$, and they satisfy the required
boundary condition (in particular, $v = v'$ along~$\partial M$).
If $\s' \in O_S$, then we can also assume that~$v'$ does not coincide with~$\nu_S$ at any point of~$S$.
We can homotope~$v'$ relative to~$\partial M$ such that $v|_S = v'|_S$. After perturbing $v'$ away from~$S$,
the set of points~$p$ where $v_p = -v'_p$ is a 1-manifold~$c$.
We orient~$c$ by taking the intersection of the submanifolds~$v(M)$ and~$-v'(M)$ of the total space of the
unit sphere bundle~$STM$, and projecting it to~$M$.
Then the homology class of~$c$ is Poincar\'e dual to $\s - \s'$.
Since $c \cap S = \emptyset$, we obtain that
\[
\langle\, \s -\s', \a \,\rangle = \# (c \cap S) = 0.
\]
In the opposite direction, suppose that $\langle\, \s -\s', \a \,\rangle = 0$. Note that $\s - \s'$ is the obstruction
to homotoping~$v$ to~$v'$ on the 2-skeleton of~$M$ through unit vector fields relative to~$\partial M$.
Hence we can homotope~$v'$ through unit vector fields relative to~$\partial M$ until $v|_S = v'|_S$.
This~$v'$ is then a representative of~$\s'$ that never agrees with~$-\nu_S$, and so~$\s' \in O_S$.

Next, we prove that if for $\s \in O_S$ and $\s' \in \spinc(M,\g)$, then we have
\[
\langle\, \s - \s', \a \,\rangle = k > 0,
\]
then $\SFH(M,\g,\s') = 0$ (in other words, $\s' \not\in \supp(M,\g)$). Represent~$\s$ by a
unit vector field~$v$, as above. After homotoping~$v$, we can assume that there is a thin regular
neighbourhood~$N$ of~$\partial M$ such that $v_p = \nu_S$ for every $p \in S \setminus N$.
Using the flow of~$v$, we obtain a product neighbourhood $P = (S \setminus N) \times [-\varepsilon,\varepsilon]$
of~$S \setminus N$ such that $v = \partial/\partial t$ on~$P$, where~$t$ is the coordinate on~$[-\varepsilon, \varepsilon]$.
Next, we construct a decomposing surface~$S'$ by attaching~$k$ small compressible tubes to~$S$, each of which lies in~$P$,
and such that each has a single hyperbolic and a single elliptic tangency with the foliation of~$P$ by the surfaces~$(S \setminus N) \times \{t\}$
for $t \in [-\varepsilon,\varepsilon]$. At the hyperbolic tangency, the orientation of the tube is opposite to that of the leaf,
while at the elliptic tangency the orientations agree.

Let $\t \in O_{S'}$, and choose a unit vector field~$w$ representing~$\t$ such that $w_p = (\nu_{S'})_p$ for every $p \in S' \setminus N$,
and such that $v_p = w_p$ for every $p \in S \cap N = S' \cap N$. Furthermore, we can assume that~$w$ is chosen such that
the set $\{\, p \in M \colon v_p = -w_p \,\}$ is a 1-manifold~$c$, oriented as above, representing~$\s - \t$. Then $|c \cap S'| = k$,
as there is exactly one intersection point in each tube in $S' \setminus S$ at the hyperbolic tangency
with the horizontal foliation on~$P$. Since each such tangency~$p$ is hyperbolic, $v_p$ is positively normal to the
foliation, and $w_p = \nu_{S'}$, the intersection sign is positive. Hence
\[
\langle\, \s - \t, [S'] \,\rangle = \#(c \cap S') = k.
\]
Of course, $[S] = [S']$, hence $\langle\, \s - \t, [S'] \,\rangle = \langle\, \s - \s', [S'] \,\rangle$.
So $\langle \t - \s', [S'] \,\rangle = 0$, and by the previous part, $\s' \in O_{S'}$ as well.

Let $(V,\nu)$ be the result of decomposing $(M,\g)$ along~$S'$. Since~$S'$ is compressible,
the sutured manifold~$(V,\nu)$ is not taut, and $\SFH(V,\nu) = 0$
(note that $(M,\g)$ is taut, so $M$ is irreducible, and hence so is~$V$).
But, according to~\eqref{eqn:decomp},
\[
\SFH(V,\nu) \cong \bigoplus_{\s' \in O_{S'}} \SFH(M,\g, \s').
\]
This implies that $\SFH(M,\g,\s') = 0$ for every $\s' \in O_{S'}$,
which means $\SFH(M,\g,\s') = 0$ for  every $\s' \in \spinc(M,\g)$
for which $\langle\, \s - \s', \alpha \,\rangle = k$.

Since $(M,\g)$ is taut, $\SFH(M,\g) \neq 0$. Notice that by~\eqref{eqn:decomp},
\[
O_S \cap \supp(M,\g) \neq \emptyset;
\]
let~$\s$ be an arbitrary element.
Pick an identification of $\spinc(M,\g)$ with $H_1(M)$. By the above, if $\s' \in \supp(M,\g)$, then
$\a(\s) \le \a(\s')$. Furthermore,
we have equality if and only if $\s \in O_S$.
So $\s' \in \spinc(M,\g)$ pairs minimally with~$\a$ if and only if~$\s' \in O_S$.
\end{proof}

We can now state the two results of the third author which we need in the proof.
The first result, which is an extension of~\cite[Proposition~4.13]{Ju10} to
balanced sutured manifolds that are not necessarily strongly balanced,
describes  how sutured Floer homology behaves under decomposition
along a nice decomposing surface.

\begin{theorem} \label{thm nice}
Let $(M,\ga)$ be a taut balanced sutured manifold, and let $\al \in H^1(M)$.
Suppose that $S$ is a nice decomposing surface that is dual to $\al$.
If the result of the sutured manifold decomposition  $(M,\ga) \leadsto^S (M',\ga')$ is taut, then
\[
\SFH(M',\ga') \cong \SFH_{\al}(M,\ga).
\]
\end{theorem}

\begin{proof}
According to~\cite[Theorem~1.3]{Ju08}, if $(M,\g)$ is a balanced sutured manifold and
$(M,\g) \rightsquigarrow^S (M',\g')$ is a sutured manifold decomposition along a nice
decomposing surface~$S$, then
\begin{equation} \label{eqn:decomp}
\SFH(M',\g') \cong \bigoplus_{\s \in O_S} \SFH(M,\g,\s).
\end{equation}
By definition,
\[
\SFH_\al(M,\ga) =\bigoplus_{\{\fs\in \spin^c(M,\ga) \colon \a \text{ pairs minimally with } \s\}}\SFH(M,\ga,\fs),
\]
so the result follows from Lemma~\ref{lem:min}, which states that
\[
O_S = \{\,\fs\in \spin^c(M,\ga) \,\colon\, \a \text{ pairs minimally with } \s \,\}.
\]
\end{proof}

In the following theorem, we put together~\cite[Proposition~9.4]{Ju06} and~\cite[Theorem~9.7]{Ju08}.

\begin{theorem} \label{thm product}
An irreducible balanced sutured manifold $(M,\ga)$ is a product if and only if
$\SFH(M,\ga) \cong \bZ$.
\end{theorem}

A sutured manifold admits a depth~0 foliation if and only if it is a product.
Indeed, then the holonomy group of each leaf is trivial, so, by Reeb stability,
$M$ is a fibre bundle over~$I$.
We thus obtain the following corollary to  Theorem \ref{thm product}
that we alluded to in the introduction.

\begin{corollary} \label{cor product}
An irreducible balanced sutured manifold $(M,\g)$ admits a taut depth~0 foliation
if and only if $\SFH(M,\g) \cong \Z$.
\end{corollary}

Combining the above two theorems, we now obtain the following (recall that well-groomed
and nice decomposing surfaces were defined in Definition~\ref{def:groomed}).

\begin{proposition}  \label{prop:decomposition}
Let $(M,\ga)$ be a balanced sutured manifold and $\al \in H^1(M)$.
If
\[
\SFH_\a(M,\g) \cong \Z,
\]
then~$\a$ is dual to a nice decomposing surface~$S$
such that the decomposition $(M,\ga) \leadsto^S (M',\ga')$ yields a product sutured manifold $(M',\ga')$.
\end{proposition}

\begin{proof}
By part~\eqref{it:3} of Lemma~\ref{lem:groomedandnice}, the class $\al$ is dual to
a well-groomed decomposing surface $S'$.
We denote by $S$ the union of the components of~$S'$ that touch~$R_-$.
Note that~$S$ has no closed components.
It follows from Theorem~\ref{mainthmtechnical} that~$S$ is also dual to~$\a$.
By part~\eqref{it:4} of Lemma~\ref{lem:groomedandnice},
the surface~$S$ is also well-groomed. Since $S$ has no closed components,
it follows from part~\eqref{it:5} of Lemma~\ref{lem:groomedandnice} that we can perturb~$S$ into a
nice surface which, by a slight abuse of notation, we also denote by~$S$.

We now perform the decomposition of $(M,\ga)$ along $S$, and call the result~$(M',\ga')$.
By part~\eqref{it:1} of Lemma \ref{lem:groomedandnice}, we know that $(M',\g')$ is taut.
It now follows from Theorem~\ref{thm nice} and our assumption on $\a$ that
\[
\SFH(M',\ga') \cong \SFH_\al(M,\ga) \cong \bZ.
\]
 Therefore, by Theorem~\ref{thm product}, this means that $(M',\ga')$ is a product.
\end{proof}

\subsection{Summary of results on foliations}
We now cite two results from the first author~\cite{Al13}.
We start with the following lemma that is the content
of~\cite[Lemma~4.4]{Al13}, together with \cite[Remark~7]{Al13}.

\begin{lemma} \label{lemma Gabai's truncation}
Let $(M,\ga)$ be a connected sutured manifold.
Suppose that $\mF$ is an indecomposable depth one foliation on $(M,\ga)$.
Then there exists a decomposing surface~$S$ with
\[
PD([S]) = \la(\mF) \in H^1(M)
\]
such that the  surface decomposition $(M,\ga) \rightsquigarrow^S (M',\ga')$ gives rise to a product sutured manifold.
\end{lemma}

For the reader's convenience, we give a quick outline of the proof.
Let~$\mF$ be an indecomposable depth one foliation.
The idea is that by ``truncating'' an arbitrary noncompact leaf of $\mF$
near $R_- \cup R_+$, we obtain a surface~$S$ giving a product decomposition.
In a bit more detail: as we mentioned already
in Section~\ref{section:deffoliation}, it was shown by
Cantwell and Conlon~\cite[Lemma~11.4.4]{Foli2} that there exists a fibre bundle
$M\setminus (R_- \cup R_+) \to S^1$ with fibres being the leaves of~$\mF$.
So we can take a leaf and remove its ends (which is equivalent to removing a neighbourhood of~$R(\ga)$).
This leaves us with a fibration over~$S^1$ whose fibres are compact surfaces,
and by removing one of these surfaces,  we obtain a product.
Finally, note that if~$c$ is a loop in~$M$, and $\langle \la(\mF), [c] \rangle$ denotes the signed intersection number
of~$c$ with a noncompact leaf~$L$, then $\langle PD( [S]),[c] \rangle=\langle \la(\mF),[c]\rangle$. It thus follows that
\[
PD([S]) = \la(\mF)\in H^1(M).
\]
This concludes the sketch of the proof of Lemma~\ref{lemma Gabai's truncation}.

\begin{lemma} \label{lemma:Al13lemC} \cite[Lemma~C]{Al13}.
Suppose $(M,\gamma)$ is a connected sutured manifold.
Let $(M,\ga) \leadsto^S (M',\ga')$ be a decomposition along a decomposing surface~$S$ such that~$(M',\ga')$ is a product.
Then there exists a depth one foliation~$\mF$ on~$(M,\ga)$ such that
the components of~$R_\pm$ are the only compact leaves, and such that~$[S]$ is dual to
$\la(\mF) \in \hom(\pi_1(M),\Z) \cong H^1(M)$.
\end{lemma}

The lemma is proved in \cite[Section~3.2]{Al13} by generalising Gabai's construction of depth one foliations to surface decompositions that are not necessarily well-groomed. For more details on Gabai's original construction, see~\cite[Theorem\,5.1]{Ga83}.

The key step in the proof of Lemma~\ref{lemma:Al13lemC}
is that we spin the decomposing surface~$S$ along~$R(\ga)$ so that~$S$ becomes a leaf of an indecomposable depth one foliation~$\mF$ of~$(M,\ga)$.
We now illustrate this spinning construction via a few simple examples.
Consider the fibration $S^1 \times I \to S^1$.
We can obtain the Reeb foliation~$\mathcal{R}$ of the annulus~$S^1 \times I$ from this product fibration
by ``spinning'' the ends of the fibres around the two boundary circles~$S^1 \times \{0\}$ and~$S^1 \times \{1\}$ in the same direction,
and adding the boundary circles as compact leaves.
The foliation~$\mathcal{R}$ is not taut, as there is no properly embedded arc or closed curve in~$S^1 \times I$ that intersects every leaf transversely.
However, if we spin the fibres of $S^1 \times I \to I$ in the opposite direction, then we get a taut foliation.

Our example of spinning in two-dimensions can easily be generalized to three-dimensions. For an analogous non-taut example, start from the fibre bundle~$S^1 \times D^2 \to S^1$, and construct the Reeb foliation of the solid torus by spinning the $D^2$-fibres around the boundary torus. For a taut example, start with the same fibre bundle, and let~$A_k$ be the shorter arc of~$S^1$ lying between~$e^{k\pi i/2}$ and~$e^{(k+1)\pi i/2}$ for $k \in \{1,\dots,4\}$.
We spin each interval~$\{x\} \times A_k$ in the boundary of the fibre~$\{x\} \times D^2$
along~$S^1 \times A_k$ in the ``positive'' direction if~$k$ is even, and in the ``negative'' direction if~$k$ is odd. This is often referred to as the ``stacking of chairs'' example, and it gives a taut, depth one foliation of the solid torus with four parallel longitudinal sutures, see~\cite[Example~5.1]{Ga84}.


\subsection{Proof of Theorem~\ref{mainthm}} Suppose $(M, \g)$ is a connected,
irreducible balanced sutured manifold, and let $\a \in H^1(M)$.
We need to show that
\[
\SFH_\a(M,\g) \cong \Z
\]
if and only if there exists
an indecomposable taut depth one foliation~$\mF$ with~$\l(\mF) = \a$.

First, suppose that  $\SFH_\a(M,\g) \cong \Z$.
Then, in particular, $\SFH(M, \g) \ne 0$, which implies by~\cite[Proposition~9.18]{Ju06}
that $(M,\g)$ is taut.
By Lemma~\ref{prop:decomposition}, it follows that~$\a$ is dual to a nice
decomposing surface~$S$
such that the decomposition
\[
(M,\ga) \leadsto^S (M',\ga')
\]
yields a product sutured
manifold~$(M',\ga')$. It follows from Lemma~\ref{lemma:Al13lemC}
that we can construct an indecomposable taut depth one foliation~$\mF$ with
\[
\la(\mF) = PD([S]) \in H^1(M).
\]

Conversely, suppose that there exists an indecomposable taut depth one foliation~$\mF$ of~$(M, \g)$ such that
$\la(\mF) = \a$.
It follows from work of Gabai~\cite[Theorem~2.12]{Ga83} that $(M,\g)$ is  taut.
By Lemma~\ref{lemma Gabai's truncation},
there exists a decomposing surface~$S$ with
\[
PD([S]) = \la(\mF) \in H^1(M)
\]
such that the surface decomposition $(M,\ga) \rightsquigarrow^S (M',\ga')$
gives rise to a product sutured manifold.
It now follows from Theorems~\ref{thm nice} and~\ref{thm product} that
\[
\SFH_\al(M,\ga) \cong \SFH(M', \ga') \cong \bZ.
\]
This concludes the proof of Theorem~\ref{mainthm}.

\subsection{Remark} \label{remark: one direction}
Sutured Floer homology of a sutured manifold is defined only if each component of the boundary contains at least one suture. The requirement $H_2(M)=0$ in~\cite[Theorem~A]{Al13} is a direct consequence of this fact: it ensures that for any properly embedded surface in $M$,
its homology class in $H_2(M, \bdd M)$ is unchanged if we discard all of its closed components.

If there is a decomposition $(M,\ga) \leadsto^S (M',\ga')$ such that $(M,\ga)$ is
balanced and $(M',\ga')$ is a product, then it is clear that $S$ cannot have closed components.
This means that if we know that there is such a decomposition, we can identify an extremal $\spin^c$-structure such that
$\SFH(M,\ga,\fs) \cong \bZ$ without the assumption $H_2(M) = 0$.

However, conversely, if there is such an extremal $\spin^c$-structure, we know there exists a surface decomposition giving a product according to the decomposition formulas~\cite[Theorem~1.3]{Ju08} and~\cite[Corollary~4.15]{Ju10}, but only if we also know that the relevant homology class
can be represented by a well-groomed decomposing surface
\emph{without closed components} (which we then perturb to a nice decomposing surface).
Such a representative always exists  when $H_2(M)=0$:
take the surface provided by part~\eqref{it:3} of~Lemma~\ref{lem:groomedandnice}, and discard
its closed components.

The purpose of this paper, contained in Proposition~\ref{mainprop}, is to show that any cohomology
class~$\a \in H^1(M)$ that pairs strictly minimally with an extremal $\spin^c$-structure, and for which
$\SFH_\a(M,\ga,\fs) \cong \bZ$, can be represented by a nice decomposing surface without closed components.

\section{fibred 3-manifolds with non-empty toroidal boundary}\label{section:fibred}

In this final section, we provide a method for deciding whether or not an oriented 3-manifold with non-empty boundary
fibres over $S^1$. We furthermore show how to  determine all the fibred classes in $H^1(M)$.
Recall that a class
\[
\a \in H^1(M) \cong \hom(\pi_1(M),\Z)
\]
is \emph{fibred} if there exists a fibration, i.e. a locally trivial fibre bundle,
$p\colon M \to S^1$ such that
\[
\a = p_* \colon \pi_1(M) \to \pi_1(S^1) = \Z.
\]
Note that for any boundary component of~$M$, the map~$p$ also restricts to a fibration over~$S^1$.
It follows that all boundary components are tori,
and that the restriction of any fibred class to any boundary component is non-zero.

We will now show that sutured Floer homology detects whether
a given first cohomology class on an irreducible 3-manifolds with non-empty toroidal boundary is fibred.
In order to state our result, we need one more definition. Let~$M$ be a $3$-manifold such that all the boundary components are tori.
A \emph{decoration} of~$M$ is a set of non-separating unoriented curves $\{c_1,\dots,c_k\}$ on~$\partial M$
such that each component of~$\partial M$ contains precisely one such curve.
To a decoration $c = \{c_1,\dots,c_k\}$ of~$M$, we associate a set of sutures~$\g(c)$ by taking a
regular neighbourhood of $\partial N(c_i)$ in ~$\partial M$,
where~$N(c_i)$ is a regular neighbourhood of~$c_i$ in~$\partial M$.
Note that $(M,\gamma(c))$ is a strongly balanced sutured manifold.
For each $\spinc$-structure~$\s$, we write
\[
\SFH(M,c,\s) = \SFH(M,\gamma(c),\s),
\]
and for each $\a \in H^1(M)$, we write
\[
\SFH_\a(M,c) = \SFH_\a(M,\gamma(c)).
\]
We can now formulate the following theorem.

\begin{theorem}\label{thm:detectfibredphi}
Let~$M$ be a connected, irreducible $3$-manifold with non-empty toroidal boundary.
Let $\a \in H^1(M)$ be a class such that the restriction of~$\a$ to any boundary component is non-zero.
We then pick a decoration $c = \{c_1,\dots,c_k\}$ of~$M$ such that $\a(c_i) \ne 0$ for every $i \in \{1,\dots,k\}$.
Then the following statements are equivalent:
\bn
\item~\label{it:Z} $\SFH_\a(M,c) \cong \Z$.
\item~\label{it:foli} There is an indecomposable taut depth one foliation $\mF$ on $(M,\g(c))$ with $\l(\mF)=\a$.
\item~\label{it:fibred} The class~$\alpha$ is fibred.
\en
\end{theorem}

\begin{proof}
The equivalence of~\eqref{it:Z} and~\eqref{it:foli} is a special case of Theorem~\ref{mainthm}.

Now we prove the equivalence of~\eqref{it:foli} and~\eqref{it:fibred}.
First, suppose that the map $p \colon M \to S^1$ is a fibration in the class~$\a$;
we denote by $\mF_p$ the corresponding depth~$0$ cooriented foliation of~$M$.
Since $\a(c_i) \neq 0$, we can arrange that~$c_i$ is transverse to~$\mF_p$ for
every $i \in \{1,\dots, k\}$. We view~$M$ as a manifold with right angle corners
along~$\partial \g(c)$, and the leaves of $\mF_p$ also have corners at~$\partial \g(c)$.
We denote by $A_i^\pm$ the part of~$R_\pm$ lying
in the component of~$\partial M$ containing~$c_i$. Then $A_i^\pm$ is an annulus,
and~$\mF_p$ restricts to~$A_i^\pm$ as a fibration with fibre~$I$. We glue the
products $A_i^\pm \times I$ to~$M$ via identifying $A_i^\pm \times \{0\}$ with~$A_i^\pm$
for every $i \in \{1,\dots, k\}$, and call the resulting 3-manifold~$M'$.
The orientation of~$M$ induces an orientation of~$M'$. Furthermore, let
\[
\g' := \g(c) \cup \bigcup_{i=1}^k (\partial A_i^\pm) \times I.
\]
We extend~$\mF_p$ to a foliation~$\mF$ of~$(M',\g')$ such that $A_i^\pm \times \{1\}$ is a compact leaf oriented as~$\pm \partial M'$
for every $i \in \{1,\dots,k\}$.
Furthermore, each leaf of the restriction of $\mF$ to $A_i^\pm \times [0,1)$ is diffeomorphic to $I \times \R_+$,
the interval $I \times \{0\}$ matches with a fibre of $\mF_p|_{A_i^\pm}$, and the arcs $I \times \{t\}$
spiral to $A_i^\pm\times \{1\}$ as $t \to \infty$. We can either spiral clockwise or counterclockwise,
exactly one of these will make the resulting foliation cooriented.
In particular, the coorientability of the foliation implies that $\mF|_{\g'}$ has no Reeb components.
There is a diffeomorphism $d \colon (M',\g') \to (M,\g(c))$ close to the obvious retraction from~$M'$ to~$M$,
and~$d(\mF)$ is an indecomposable taut depth one foliation of~$(M,\g(c))$ with $\l(\mF) = \a$.

In the opposite direction, assume that~$\mF$ is an indecomposable taut depth one foliation of~$(M,\g(c))$
such that $\l(\mF) = \a$. Again, we view~$M$ as a manifold with corners along~$\partial \g(c)$.
We claim that, in this case, $\mF$ is of the form described in the previous paragraph. More precisely,
we will show that there is a collar neighbourhood~$N$ of $R = R_+ \cup R_-$ such that the restriction
of~$\mF$ to~$\overline{M \setminus N}$ is a fibration. Indeed, since~$\mF$ is indecomposable,
the union of the compact leaves of~$\mF$ is~$R$, and hence $M \setminus R$ is the union of the depth one
leaves. The holonomy of each depth one leaf is trivial (cf.~Definition~\ref{def:holonomy}),
otherwise there would be leaves limiting on depth one leaves,
and hence Reeb stability implies that~$\mF$ restricts to a fibration of~$M \setminus R$
with leaf space~$S^1$ as~$M$ is connected. In particular, in each component~$\g_0$ of~$\g(c) \setminus R$, the foliation~$\mF$
restricts to a cooriented fibration over~$S^1$. This determines the foliation~$\mF|_{\g_0}$ uniquely up to isotopy
in the annulus~$\g_0$: each leaf is diffeomorphic to~$\R$, and spirals clockwise to~$R_+ \cap \g_0$ and counterclockwise
to~$R_- \cap \g_0$.

Let~$A$ be one of the two components of~$R$ that meets~$\g_0$. Since~$M$ has toroidal boundary,
$A$ is an annulus. We now  take an identification of $A$ with~$S^1 \times I$. Furthermore, let~$A \times I$ be a collar neighbourhood
of~$A$ in~$M$ such that for every point $(x,t) \in A \approx S^1 \times I$,
the arc $\{(x,t)\} \times I$ is transverse to~$\mF$.
Furthermore, we can assume that there is an $\eps > 0$ such that
\[
\nu = S^1 \times \{0\} \times [0,\eps]
\]
is a collar neighbourhood of~$A \cap \g_0$ in~$\g_0$ such that
$\partial \nu \setminus A = S^1 \times \{0\} \times \{\eps\}$ is a curve transverse to~$\mF$.
Then, for every $(x,t,s) \in S^1 \times I \times [0,\eps] \approx A \times [0,\eps]$, let 
\[
\psi(x,t,s) = (x,t,\eta(x,t,s)) \in S^1 \times I \times I,
\]
where $\eta$ denotes the third component of~$\psi$, and $t \mapsto \psi(x,t,s)$
is the unique curve of the above form tangent to~$\mF$ such that $\psi(x,0,s) = (x,0,s)$
(the lift of the curve $t \mapsto (x,t) \in A$ starting from $(x,0,s)$).
The map~$\psi$ is unique and well-defined if~$\eps$ is sufficiently small.
Then~$\psi$ defines a collar neighbourhood~$N_A$ of~$A$ on which~$\mF$ is the product $\mF|_\nu \times I$,
hence consists of leaves $\R \times I$ spiralling to~$A$, exactly as in the proof of \eqref{it:fibred} $\Rightarrow$ \eqref{it:foli}.
This way, we obtain a collar neighbourhood of each component of~$R$,
and the union of these we denote by~$N$. On~$\partial(\overline{M \setminus N})$, the foliation~$\mF$ restricts to
a fibration by circles. Finally, every leaf~$L$ of~$\mF|_{\overline{M \setminus N}}$ is compact,
as~$L$ is a closed subset of a depth one leaf of~$\mF$, so if a sequence~$(p_n)$ in~$L$ had no limit point in~$L$,
then it would have a subsequence converging to~$R$ (as this is the union of the depth~$0$ leaves),
contradicting the fact that each~$p_n$ lies outside the neighbourhood~$N$ of~$R$.

In this section, we are mostly interested in the equivalence of~\eqref{it:Z} and~\eqref{it:fibred}.
Note that this can also be proved in a  more direct way using some of the methods of this paper, not pertaining to foliations.
First, suppose that~\eqref{it:Z} holds. Then, as in the proof of Theorem~\ref{mainthm},
the class~$\a$ is dual to a nice decomposing surface~$S$ such that the decomposition
\[
(M,\g(c)) \rightsquigarrow^S (M',\g')
\]
yields a product sutured manifold~$(M',\g')$. Furthermore, since~$\a(c_i) \neq 0$ for every $i \in \{1,\dots,k\}$,
and because~$S$ is dual to~$\a$, we can assume that $\partial S$ is transverse to~$c$, and hence intersects
each component of~$\g(c)$ in essential arcs. For each component~$T$ of~$\partial M$, the set~$S \cap T$ consists
of parallel oriented curves that divide~$T$ into parallel annuli $A_1, \dots, A_l$. After
decomposing~$(M,\g(c))$ along~$S$, there will be exactly one suture in each~$A_i$, and we can take the~$A_i$
themselves to be the resulting components of~$\g'$.
So $(M',\g') \approx (S \times I, \partial S \times I)$, and we obtain~$M$ by gluing $S \times \{0\}$ to $S \times \{1\}$.
In particular, $M$ is a fibre bundle over~$S^1$ with the surface~$S$, the dual to~$\a$, being one of the fibres.

Finally, we show that~\eqref{it:fibred} implies~\eqref{it:Z}. Assume that~$M$ fibres over~$S^1$, and that~$S$ is a fibre dual
to~$\a$. Again, we can suppose that~$c$ is transverse to~$S$. Note that~$S$ is a nice
decomposing surface in~$(M,\g(c))$. As above, the result of decomposing~$(M,\g(c))$
along~$S$ gives a sutured manifold~$(M',\g')$ diffeomorphic to~$(S \times I, \partial S \times I)$, which is taut.
So Theorem~\ref{thm nice} implies that
\[
\SFH_\a(M,\g(c)) \cong \SFH(M',\g') \cong \Z.
\]
This concludes the proof.
\end{proof}

Theorem \ref{thm:detectfibredphi} says, in particular, that given a connected,
irreducible 3-manifold with non-trivial toroidal boundary,
we can use sutured Floer homology to detect whether or not a given class $\a\in H^1(M)$ is fibred.
Nonetheless, the result is perhaps slightly unsatisfactory, as it does not say how one can determine
all fibred classes from sutured Floer homology at once.
Just using Theorem \ref{thm:detectfibredphi}  it is also not clear how one can  decide whether
there is a fibred class at all.

 The issue is that in general there is no decoration which `works for all non-zero $\alpha\in H^1(M)$'. More precisely, if we fix a
decoration~$c=\{c_1,\dots,c_k\}$ of~$M$, then for some classes~$\a \in H^1(M)$, we might have $\a(c_i) = 0$ for some $i \in \{1,\dots,k\}$.
Now let~$T$ be a component of~$\partial M$ with decoration $c_i$ such that $\a(c_i)=0$. If~$S$ is then a surface dual to~$\a$ in~$M$,
then $S \cap \partial T$ consists of say~$l$ parallel curves, all isotopic to~$c_i$.
Put~$S$ in a position where these~$l$ curves are all parallel to~$c_i$.
When~$S$ is a fibre of a fibration, and we decompose $(M,\g)$ along such an~$S$, we obtain $(M',\g')$, where $M' = S \times I$,
but on the component of $\partial S \times I$ containing~$c_i$, there will be~$3$ parallel components of~$\g'$.
Hence, according to \cite[Proposition~9.2]{Ju10}, if~$S \neq D^2$, then $\SFH(M',\g') \cong \left(\Z^2\right)^{\otimes s(\a)}$, where
\[
s(\a) = |\{\, i \in \{1,\dots,k\} \,\colon\, \a(c_i) = 0 \,\}|.
\]
In other words, if we first fix the decoration~$c$, condition~\eqref{it:fibred} does not imply~\eqref{it:Z},
unless~$s(\a) = 0$. On the other hand, if we know that $\SFH_\a(M,\g(c)) \cong \left(\Z^2 \right)^{\otimes s(\a)}$,
we cannot use Theorem~\ref{mainthmtechnical} to obtain a nice decomposing surface dual to~$\a$ that gives a
taut decomposition, as we might be stuck with closed components. However, if we put a restriction on the
homology of~$M$, we obtain the following.

\begin{proposition} \label{prop:fibred}
Let~$M$ be a connected, irreducible 3-manifold with non-empty toroidal boundary,
and such that the map $H_2(\partial M;\Q) \to H_2(M; \Q)$ is surjective.
Pick a decoration~$c = \{c_1,\dots,c_k\}$ of~$\partial M$ such that no component of $\partial M \setminus c$
is compressible, and let $\a \in H^1(M)$. Then the following are equivalent:
\begin{enumerate}
\item \label{it:i} The class~$\a$ is fibred.
\item \label{it:ii} $\SFH_\a(M,\g(c)) \cong \left(\Z^2 \right)^{s(\a)}$.
\end{enumerate}
\end{proposition}

Note that the assumption that  $H_2(\partial M;\Q) \to H_2(M; \Q)$ is surjective
is equivalent to saying that $M$ is the exterior of a link in a rational homology sphere.

\begin{proof}
First, suppose that the class~$\a$ is fibred. Then let~$S$ be a fibre of a fibration such that~$[S]$
is dual to~$\a$. Isotope~$S$ such that it is transverse to~$c$, and parallel to~$c_i$ whenever~$\a(c_i) = 0$.
We saw above that if we decompose~$(M,\g(c))$ along~$S$, then we obtain a sutured manifold~$(M',\g')$
such that $M \approx S \times I$, and $\g'$ consists of~$\partial S \times I$, except that there are three
parallel sutures for each component of~$\partial S \times I$ that contains a curve~$c_i$ for which~$\a(c_i) = 0$.
Since no component of~$\partial M \setminus \g(c)$ is compressible, $(M',\g')$ is not~$D^2 \times I$ with three
parallel sutures, which is not taut and hence has~$\SFH(M',\g') = 0$. So~\cite[Proposition~9.2]{Ju10}
and Theorem~\ref{thm nice} imply that
\[
\SFH_\a(M,\g) \cong \SFH(M',\g') \cong \left(\Z^2 \right)^{s(\a)}.
\]
This shows that~\eqref{it:i} implies~\eqref{it:ii}.

Now suppose that~\eqref{it:ii} holds. Since no component of $\partial M \setminus c$ is compressible,
the sutured manifold~$(M,\g(c))$ is tight. Hence, according to Gabai~\cite{Ga83}, there exists
a well-groomed decomposing surface~$S$ such that $PD[S] = \a$ (in particular, the result of the decomposition is taut).
As usual, we isotope~$S$ such that~$\partial S$ is transverse to~$c$, and that it is parallel to
each curve~$c_i$ for which~$\a(c_i) = 0$.
Let~$S_0$ be the union of the closed components of~$S$.
Since the map $H_2(\partial M;\Q) \to H_2(M;\Q)$ is surjective, the map
$j_* \colon H_2(M;\Q) \to H_2(M,\partial M;\Q)$ is zero, and hence $S_0$ is $0$-homologous
in~$H_2(M,\partial M;\Q)$, and so also in $H_2(M,\partial M)$ (as the latter is torsion-free).
So $S' = S \setminus S_0$ is also dual to~$\a$. Furthermore, $S'$ will also produce a taut
sutured manifold that we denote~$(M',\g')$. By Theorem~\ref{thm nice} and our assumption~\eqref{it:ii},
\[
\SFH(M',\g') \cong \left(\Z^2 \right)^{s(\a)}.
\]
If $\a(c_i) = 0$, then the component~$A_i$ of~$\partial M \setminus N(S)$ containing~$c_i$ will have three
parallel sutures in~$(M',\g')$.
According to~\cite[Proposition~9.2]{Ju10}, removing two of the three parallel sutures from~$\g'$
for each of the $s(\a)$ curves~$c_i$
for which $\a(c_i) = 0$ will result in a sutured manifold~$(M',\nu)$ with~$\SFH(M',\nu) \cong \Z$.
Hence $(M',\nu)$ is a product with $R_\pm(\nu) \approx S'$, and so~$M' \approx S' \times I$. Furthermore, $M$ is obtained by
gluing $S' \times \{0\}$ and $S' \times \{1\}$ via some diffeomorphism, and so~$\a$ is a fibred class.
\end{proof}

We will now describe a method for detecting all fibred classes (and whether there is a
fibred class at all) even when the map $H_2(\partial M;\Q) \to H_2(M;\Q)$ is not surjective,
and along the way, we will also consider real fibred classes.
We say that a class $\a \in H^1(M;\R)$ is \emph{fibred} if it can be represented by a nowhere vanishing closed 1-form.
It is well-known that, for integral classes, the two notions of fibredness agree. We now denote by
\[ \FF(M):=\{ \a \in H^1(M;\R)\,:\, \a\mbox{ is fibred}\}\]
the set of all fibred classes.

In the following,  given a decoration $c=\{c_1,\dots,c_k\}$ for
a connected, irreducible $3$-manifold with non-empty toroidal boundary, we write
\[
\CC(M,c): = \{ \a\in H^1(M;\R)\,\colon\, \a(c_i)\ne 0\mbox{ for }i=1,\dots,k\}.
\]
This set contains all classes for which $\SFH$ can detect whether they are fibred,
according to Theorem~\ref{thm:detectfibredphi}. More precisely, let
\[
\MM(M,c): = \{ \a\in \CC(M,c)\,\colon\, \SFH_\a(M,c)\cong \Z\};
\]
the integral classes in $\MM(M,c)$ are exactly the fibred classes in $\CC(M,\g)$.
By definition, $\MM(M,c) \subset \CC(M,c) \subset H^1(M;\R)$.

Before we state the next theorem, we recall that if~$M$ is a 3-manifold,
then it is a consequence of a standard Poincar\'e duality argument that
for any toroidal boundary component~$T$, we have $\rk(\im(H_1(T)\to H_1(M)))\geq 1$.

\begin{theorem}\label{thm:fibredcones}
Let $M$  be a connected, irreducible $3$-manifold such that $\partial M$ is a non-empty union of tori $T_1,\dots,T_k$, labeled such that
\[
\rk \left(\im(H_1(T_i)\to H_1(M)) \right) =
\begin{cases}
2 & \text{  for $i \in \{1,\dots,l\}$, and}\\
1 & \text{ for $i \in \{l+1,\dots,k\}$.}
\end{cases}
\]
For $i \in \{1,\dots,l\}$, we pick curves $c_i^1$, $c_i^2 \subset T_i$
such that they span a rank two subspace of $H_1(M)$,
and for $i \in \{l+1,\dots,k\}$, we pick a curve $c_i \subset T_i$ that is non-torsion in~$H_1(M)$.
Then
\[
\FF(M) = \bigcup_{\{\eps_1,\dots,\eps_l\} \in \{1,2\}^l} \MM(M,c_1^{\eps_1},\dots,c_l^{\eps_l},c_{l+1},\dots,c_k).
\]
\end{theorem}

\begin{proof}
We write
\[  \CC(M):=\{ \a\in H^1(M;\R)\,\colon\, \a|_{T_i}\ne 0\mbox{ for }i=1,\dots,k\}.\]
For any of the sets $\CC(M)$, $\CC(M,c)$, $\FF(M)$, and $\MM(M,c)$,
we use the same expression decorated with a superscript~$\Z$
to indicate their intersection with~$H^1(M)$. With these conventions,
we can reformulate the equivalence of~\eqref{it:Z} and~\eqref{it:fibred} in Theorem~\ref{thm:detectfibredphi} as
\begin{equation}
\label{equ:cfh} \FF^{\Z}(M)\cap\CC^{\Z}(M,c) = \MM^{\Z}(M,c).
\end{equation}
We denote by~$D$ the set of the~$2^l$ decorations
$\{c_1^{\eps_1},\dots,c_l^{\eps_l},c_{l+1},\dots,c_k\}$ with $\{\eps_1,\dots,\eps_l\}$ in~$\{1,2\}^l$.

We first prove the following claim.

\begin{claim}
\[  \FF^{\Z}(M)=\bigcup_{c\in D} \MM^{\Z}(M,c).\]
\end{claim}

It follows from the earlier discussion that $\FF^{\Z}(M)\subset \CC^{\Z}(M)$.
It also follows easily from the definitions that
\[  \CC^{\Z}(M)=\bigcup_{c\in D} \CC^{\Z}(M,c).\]
If we combine these observations with  (\ref{equ:cfh}), we obtain that
\[ \ba{rcl} \FF^{\Z}(M)&=&\FF^{\Z}(M)\cap \CC^{\Z}(M)\\[2mm]
&=&\FF^{\Z}(M)\cap \bigcup_{c\in D} \CC^{\Z}(M,c)\\[2mm]
&=& \bigcup_{c\in D} \left( \FF^{\Z}(M) \cap \CC^{\Z}(M,c) \right)\\[2mm]
&=& \bigcup_{c\in D} \MM^{\Z}(M,c).\ea \]
This concludes the proof of the claim.

By a rational open half-space in $H^1(M;\R) \cong H^1(M) \otimes \R$, we mean a subset that is described
by a strict rational linear inequality. By a rational cone in~$H^1(M;\R)$, we mean
the intersection of finitely many rational open half-spaces in~$H^1(M;\R)$.
Thurston~\cite{Th86} showed that~$\FF(M)$ is a (possibly empty) union of rational cones.
It is a straightforward consequence of the definitions that each~$\MM(M,c)$ is a rational cone.

Together with the above claim, we now see that
$\FF(M)$ and $\bigcup_{c\in D} \MM(M,c)$ are unions of rational cones
that agree on $H^1(M)$. It follows that the two sets are in fact the same.
\end{proof}

\begin{corollary}
With the notation of Theorem~\ref{thm:fibredcones}, we have $\FF(M,\g) \neq \emptyset$
if and only if $\SFH(M,\g(c),\s) \cong \Z$ for an extremal $\spinc$-structure $\s \in \spinc(M,\g(c))$
for one of the $2^l$ decorations $c = \{c_1^{\eps_1},\dots,c_l^{\eps_l},c_{l+1},\dots,c_k\}$
for $\{\eps_1,\dots,\eps_l\} \in \{1,2\}^l$.
\end{corollary}

\begin{proof}
We first suppose that  $\FF(M,\g) \neq \emptyset$. It then follows
immediately from Theorem~\ref{thm:fibredcones}
that, for one of the given~$2^l$ decorations~$c$, there exists an extremal $\spinc$-structure $\s \in \spinc(M,\g(c))$
such that $\SFH(M,c,\s) \cong \Z$.

Now suppose that, for one of the $2^l$-decorations~$c$, there exists an extremal $\spinc$-structure~$\s$
such that $\SFH(M,c,\s) \cong \Z$.
This implies that there exists an $\a \in H^1(M;\R)$ such that
\[ \SFH_{\a}(M,c)\cong \Z.\]
It is straightforward to see that the set
\[ \{ \b\in H^1(M;\R)\sm \{0\} \,:\, \SFH_{\b}(M,c)\cong \Z\}\]
is open in~$H^1(M;\R)$. Hence, there also exists a class $\b\in \CC(M,c)$ with $\SFH_{\b}(M,c) \cong \Z$.
It thus follows from Theorem~\ref{thm:fibredcones} that~$M$ is fibred.
\end{proof}

We finally note that if $M$ is a 3-manifold such that the boundary consists of a single torus~$T$,
then $ \rk(\im(H_1(T)\to H_1(M)))=1$. We can thus record the following special case of Theorem~\ref{thm:fibredcones}.

\begin{corollary}
Let~$M$  be a connected, irreducible $3$-manifold such that the boundary consists of a single torus~$T$.
We pick a curve~$c$ on~$T$ that is non-torsion in~$H_1(M)$.
Then
\[  \FF(M)= \MM(M,c).\]
\end{corollary}


\begin{thebibliography}{100000}
\bibitem[AN09]{AN09}
Y. Ai and Y. Ni, {\em Two applications of twisted Floer homology},
Int. Math. Res. Not.  2009, no. 19, 3726--3746.

\bibitem[Al13]{Al13}
I. Altman, {\em The sutured Floer polytope and taut depth one foliations}, Preprint (2013), to be published by Alg. Geom. Top.
\bibitem[BP01]{BP01}
R. Benedetti and C. Petronio, {\em Reidemeister –Turaev torsion of 3-dimensional Euler structures with
simple boundary tangency and pseudo-Legendrian knots}, Manuscripta Math. 106 (2001), 13--61.
\bibitem[CN85]{CN85}
C. Camacho and A. L. Neto, \emph{Geometric theory of foliations}, Birkh\"auser, 1985.
\bibitem[CC99]{Foli1}
A. Candel and L. Conlon, \emph{Foliations I},  Graduate Studies in Mathematics, 23. American Mathematical Society, Providence, RI, 1999.
\bibitem[CC03]{Foli2}
A. Candel and L. Conlon, \emph{Foliations II},  Graduate Studies in Mathematics, 60. American Mathematical Society, Providence, RI, 2003.
\bibitem[FJR11]{FJR11}
S. Friedl, A. Juh\'asz and J. Rasmussen, {\em The decategorification of sutured Floer homology},
 J. Topology  \textbf{4} (2011), 431--478.
\bibitem[Ga83]{Ga83}
D. Gabai, {\em Foliations and the topology of 3-manifolds}, J. Differential Geometry \textbf{18} (1983),
445--–503.
\bibitem[Ga84]{Ga84}
D. Gabai, {\it Foliations and genera of links}, Topology {\bf 23} (1984), 381--394.
\bibitem[Ga87]{Ga87}
D. Gabai, {\it Foliations and the topology of 3-manifolds II}, J. Differential Geom. \textbf{26} (1987), no. 3,
461--478.
\bibitem[Gh08]{Gh08}
P. Ghiggini, {\em Knot Floer homology detects genus-one fibred knots}, Amer. J. Math.  130, no. 5 (2008), 1151--1169.
\bibitem[HS97]{HS97}
P. J. Hilton and U. Stammbach, {\em A Course in Homological Algebra}, second edition, Springer
Graduate Texts in Mathematics (1997).
\bibitem[Ju06]{Ju06}
A. Juh\'asz, {\em Holomorphic discs and sutured manifolds},  Algebraic \& Geometric Topology \textbf{6} (2006), 1429--1457.
\bibitem[Ju08]{Ju08}
A. Juh\'asz, {\em Floer homology and surface decompositions}, Geometry and Topology \textbf{12:1} (2008), 299--350.
\bibitem[Ju10]{Ju10}
A. Juh\'asz, {\em The sutured Floer homology polytope},  Geometry and Topology \textbf{14} (2010), 1303--1354.
\bibitem[Ni07]{Ni07}
Y. Ni, {\em Knot Floer homology detects fibred knots}, Invent. Math. 170 (2007), no. 3, 577--608.
\bibitem[Ni09a]{Ni09a}
Y. Ni, {\em Erratum: Knot Floer homology detects fibred knots}, Invent. Math. 170 (2009), no.~1, 235--238.
\bibitem[Ni09b]{Ni09}
Y. Ni, {\em Heegaard Floer homology and fibred 3-manifolds}, Amer. J. Math. 131 (2009), no. 4, 1047--1063.
\bibitem[OS04a]{OS04a}
P. Ozsv\'ath and Z. Szab\'o, {\em Holomorphic disks and topological invariants for closed 3-manifolds}, Annals of Mathematics \textbf{159} (2004), no. 3, 1027--1158.
\bibitem[OS04b]{OS04b}
P. Ozsv\'ath and Z. Szab\'o, {\em Holomorphic disks and 3-manifold invariants: properties and applications}, Annals of Mathematics \textbf{159} (2004) no. 3, 1159--1245.
\bibitem[OS04c]{OS04c}
P. Ozsv\'ath and Z. Szab\'o, {\em Holomorphic disks and knot invariants}, Advances in Mathematics \textbf{186} (2004), no. 1,
58--116.
\bibitem[OS08]{OS08}
P. Ozsv\'ath and Z. Szab\'o, {\em Holomorphic disks, link invariants, and the multi-variable Alexander polynomial},
Algebraic \& Geometric Topology \textbf{8} (2008), no. 2,
615--692.
\bibitem[Ras03]{Ras03}
J. Rasmussen, {\em Floer homology and knot complements}, PhD Thesis, Harvard University (2003).
\bibitem[Th86]{Th86}
W. P. Thurston, {\em A norm for the homology of 3--manifolds},
Mem. Amer. Math. Soc. \textbf{59} (1986), no. 339, i--vi and
99--130.
\bibitem[Tu01]{Tu01} V. Turaev, {\em Introduction to combinatorial torsions}, Birkh\"auser Verlag, Basel, 2001.
\bibitem[Tu02]{Tu02} V. Turaev, {\em Torsions of $3$-dimensional manifolds}, Progress in Mathematics, 208. Birkh\"auser Verlag, Basel, 2002.
\end{thebibliography}
\end{document}